\documentclass[twoside, 11pt]{article}
\usepackage{geometry}
\usepackage{amsfonts}
\usepackage{cite}
\usepackage{latexsym}
\usepackage{enumerate,verbatim}
\usepackage{amsfonts}
\usepackage[ansinew]{inputenc}
\usepackage{amsmath,amsthm,amsfonts,amssymb}
\usepackage{graphics}
\usepackage{mathrsfs}
\usepackage{bm}

\usepackage{color}

\usepackage{algorithm}
\usepackage{algorithmic}
\usepackage{tikz}

 \newtheorem{theorem}{Theorem}[section]
 \newtheorem{cor}[theorem]{Corollary}
 \newtheorem{lemma}[theorem]{Lemma}
 
  \newtheorem{example}[theorem]{Example}
 \newtheorem{prop}[theorem]{Proposition}
 \theoremstyle{definition}
 \newtheorem{defi}[theorem]{Definition}
 \theoremstyle{remark}
 \newtheorem{remark}[theorem]{Remark}

 \theoremstyle{fact}
 \newtheorem{fact}[theorem]{Fact}
\numberwithin{equation}{section}

\newcommand{\od}{ { \mathcal{O}(d)} }
\newcommand{\odn}{ { \mathcal{O}(d)^{\otimes n} } }

\newcommand{\sod}{ { \mathcal{SO}(d)} }
\newcommand{\sodn}{ { \mathcal{SO}(d)^{\otimes n} } }
\newcommand{\asign}[1]{\mathrm{asign}\left(#1 \right)}

\newcommand{\be}{\begin{equation}}
	\newcommand{\bee}{\end{equation}}
\newcommand{\en}{\begin{equation*}}
	\newcommand{\een}{\end{equation*}}
\newcommand{\eqn}{\begin{eqnarray}}
	\newcommand{\eeqn}{\end{eqnarray}}
\newcommand{\bmat}{\begin{bmatrix}}
	\newcommand{\emat}{\end{bmatrix}}

\title{\bf Orthogonal Group Synchronization with Incomplete Measurements: Error Bounds and Linear Convergence of the Generalized Power Method}

\author{Linglingzhi Zhu\thanks{Department of Systems Engineering and Engineering Management, The Chinese University of Hong Kong, Shatin, N.T., Hong Kong (llzzhu@se.cuhk.edu.hk)} \and Jinxin Wang\thanks{Department of Systems Engineering and Engineering Management, The Chinese University of Hong Kong, Shatin, N.T., Hong Kong (jxwang@se.cuhk.edu.hk)} \and Anthony Man-Cho So\thanks{Department of Systems Engineering and Engineering Management, The Chinese University of Hong Kong, Shatin, N.T., Hong Kong (manchoso@se.cuhk.edu.hk)}
}

\begin{document}
\maketitle

\begin{abstract}
Group synchronization refers to estimating a collection of group elements from the noisy pairwise measurements. Such a nonconvex problem has received much attention from numerous scientific fields including computer vision, robotics, and cryo-electron microscopy. In this paper, we focus on the orthogonal group synchronization problem with general additive noise models under incomplete  measurements, which is much more general than the commonly considered setting of complete measurements.
Characterizations of the orthogonal group synchronization problem are given from perspectives of optimality conditions as well as fixed points of the projected gradient ascent method which is also known as the generalized power method (GPM). It is well worth noting that these results still hold even without generative models.
In the meantime, we derive the local error bound property for the orthogonal group synchronization problem which is useful for the convergence rate analysis of different algorithms and can be of independent interest.
Finally,  we prove the linear convergence result of the GPM to a global maximizer under a general additive noise model based on the established local error bound property. Our theoretical convergence result holds under several deterministic conditions which can cover certain cases with adversarial noise, and as an example we specialize it to the setting of the Erd\"os-R\'enyi measurement graph and Gaussian noise.

\end{abstract}


\section{Introduction}
Given some noisy pairwise measurements of $n$ unknown elements of the $d\times d$ orthogonal group $\od:= \{g \in \mathbb{R}^{d \times d} \mid g^\top g=gg^\top=I_d\}$, the  synchronization problem is to estimate those $n$ group elements efficiently and robustly. 
Concretely, the orthogonal group synchronization problem with incomplete measurements can be stated as follows. Let $G^\star:=[G_1^\star;\ldots;G_n^\star]\in \odn := \{G \in \mathbb{R}^{nd \times d} \mid G_i \in \od,\ i \in [n] \}$, where $[n] :=\{1,\ldots,n\}$ and $G_i$ denotes the $i$-th $d \times d$ block of $G$, be the underlying matrix (ground truth). The noisy incomplete pairwise measurements of a typical additive noise model are
\begin{equation}\label{eq: gsmodel1}
	C_{ij}=
	\begin{cases}
		W_{ij}\cdot (G_i^\star G_j^{\star\top}+\Delta_{ij}), & \quad \text{ if } i\neq j,\\
		W_{ii}\cdot I_d, & \quad \text{ otherwise},  
	\end{cases}
\end{equation}
where $W\in \mathbb{R}^{n \times n}$ is the symmetric adjacency matrix of the measurement graph $\mathcal{G}([n],\Omega)$ with an edge set $\Omega$, and $\Delta_{ij} \in \mathbb{R}^{d \times d}$ represents the additive noise with 
$\Delta_{ij} = \Delta_{ji}^\top$ for all $i,j\in [n]$ and $\Delta_{ii}= \bm{0}$ for all $i\in [n]$. 
By defining $A := W \otimes (1_d 1_d^\top)$, where ``$\otimes$" is the Kronecker product, the additive noise model \eqref{eq: gsmodel1} can be rewritten compactly as
\begin{align} \label{eq: gsmodel3}
	C = A \circ (G^\star G^{\star\top} + \Delta),
\end{align}
where ``$\circ$" is the Hadamard product, and $\Delta \in \mathbb{R}^{nd \times nd}$ is a symmetric matrix with its $ij$-th $d\times d$ block being $\Delta_{ij}$.
A common approach to find an estimate $\hat{G} \in \odn$ of $G^\star$ is through the maximum likelihood estimation, which is to solve the following optimization problem: 
\begin{equation}
\begin{array}{ll}
\min\limits_{G\in\mathbb{R}^{nd\times d}} &\sum\limits_{(i,j)\in\Omega,i<j}\Vert G_i G_j^\top -C_{ij} \Vert_F^2\\
\ \ \ {\rm s.t.} &G_i\in \od, \  \textrm{ for } i \in [n].
\end{array}
\end{equation}
Utilizing the fact that the variables $\{ G_i\}_{i=1}^n$ are orthogonal matrices, the above optimization problem can be further recast as the following quadratic problem,
\begin{equation}\label{Problem}\tag{QP}
\begin{array}{ll}
\max\limits_{G\in \odn } f(G):=\textrm{tr}(G^\top C G).\\
\end{array}
\end{equation}

In spite of its simple formulation, the orthogonal group synchronization problem has long gained considerable attention from various areas including computer vision \cite{chaudhury2015global}, robotics \cite{rosen2020certifiably}, and cryo-electron microscopy \cite{singer2018mathematics}. In addition, there are plenty of synchronization problems on subgroups of the orthogonal group. For example, rotation averaging and structure from motion (special orthogonal group $\mathcal{SO}(d)$) \cite{dellaert2020shonan,arie2012global,hartley2013rotation}, angular synchronization (special orthogonal group $\mathcal{SO}(2)$) \cite{singer2011angular}, joint alignment ($\mathcal{Z}_m$ group) \cite{chen2018projected}, community detection ($\mathcal{Z}_2$ group) \cite{abbe2017community,wang2020nearly}, and multi-graph matching (permutation group $\mathcal{P}(d)$) \cite{pachauri2013solving}. Hence, studying problem \eqref{Problem} not only brings significant impact to the orthogonal group synchronization, but also provides rather deep insights for other related problems.

In general, problem \eqref{Problem} is NP-hard with the max-cut problem ($d=1$) as a special case. Nevertheless, it pertains to the well-known quadratic program with quadratic constraints (QPQC), where the semidefinite programming (SDP) relaxation technique \cite{luo2010semidefinite} can be possibly employed to lead to polynomial-time algorithms (e.g., the interior point method) to obtain a feasible solution with provable approximation ratio of the optimal value. Typically, the SDP relaxation is not tight, which implies that the feasible point obtained is not necessarily globally optimal. Surprisingly, when focusing on problem \eqref{Problem} with $p=1$, the SDP relaxation has been proven to be tight if the noise strength is relatively small. More specifically, in \cite{zhang2019tightness} and \cite{ling2020solving}, the authors prove that under Gaussian noise setting, the optimal solution to problem \eqref{Problem} can be exactly recovered through the SDP relaxation approach if the noise intensity satisfies $\sigma = O(\frac{n^{1/4}}{d^{5/4}})$ and $\sigma = O(\frac{n^{1/4}}{d^{3/4}})$, respectively. Very recently, the author in \cite{ling2020improved} further improves the bound to $O(\frac{\sqrt{n}}{\sqrt{d}(\sqrt{d}+\sqrt{\log n})})$ using the leave-one-out technique as a key tool. However, a major drawback of the SDP relaxation is that it fails to scale well with $n$ as the number of variables after the SDP relaxation turns to $(n^2-n) d^2$ instead of that of $nd^2$ in problem \eqref{Problem}. Even though there are several attempts in designing custom-developed algorithms for certain kinds of SDP \cite{wen2012block,boumal2016non,wang2013exact}, the corresponding theories as well as computational efficiency are not entirely satisfactory.

In addition to the aforementioned SDP relaxation method, one can also apply the light-weight first-order method, namely the generalized power method (GPM), to resolve problem \eqref{Problem} \cite{ling2020improved,journee2010generalized,boumal2016nonconvex,liu2017estimation}, which can be seen as a projected gradient method to the product manifold $\odn$ in the Euclidean space $\mathbb{R}^{nd\times d}$. There are fruitful theoretical results of the GPM when employed as a workhorse for problem \eqref{Problem} as well as other relevant problems \cite{boumal2016nonconvex,liu2017estimation,zhong2018near,shen2020complete,ling2020improved}. For instance, for the phase synchronization ($\mathcal{SO}(2)$) problem, the work \cite{boumal2016nonconvex} proves that under the additive noise model, the GPM with carefully designed spectral initialization, which is close enough to the ground truth, would converge to a global optimum. Further, the authors in \cite{liu2017estimation,zhong2018near} show that the convergence rate is actually linear under even certain improved noise bounds. For general group synchronization problems, a unified framework for group synchronization problems over subgroups of the orthogonal group (e.g., orthogonal group $\mathcal{O}(d)$, special orthogonal group $\mathcal{SO}(d)$, $\mathcal{Z}_m$ group, permutation group $\mathcal{P}(d)$) has been proposed in \cite{liu2020unified} showing that the estimation error of the iterates (i.e., the distance between the iterates and the ground truth) generated via GPM decreases geometrically under certain conditions on the subgroup, the measurement graph, the noise model and the initialization regardless of the convergence of the iterates. To further clarify the convergence property of the GPM left in \cite{liu2020unified} for problem \eqref{Problem}, the work \cite{ling2020improved} establishes the linear convergence rate of the GPM to a global optimum under complete measurements with the Gaussian noise strength $\sigma = O(\frac{\sqrt{n}}{\sqrt{d}(\sqrt{d}+\sqrt{\log n})})$ following a similar routine in \cite{zhong2018near}, which is comparable with the noise bound for the SDP relaxation method. However, the analyses of the linear convergence rate therein heavily rely on the iterates of the GPM rather than the landscape of problem \eqref{Problem}, and the techniques used there can be difficult when applied to analyze the general additive noise models and also the incomplete measurements case.


In this paper, we study the problem \eqref{Problem} 
and our contributions are the following three perspectives:
\begin{itemize}
\item First, we give a comprehensive characterization of problem \eqref{Problem} without the generative model on $C$, including optimality conditions and the relationship among different notions of points we are interested in, e.g., fixed points of the GPM, global maximizers and critical points, which would help us understand the problem including but not limited to the orthogonal group synchronization and also design suitable algorithms.
    \item Second, we derive the local error bound property of problem \eqref{Problem} under the generative model with incomplete measurements and general noise given by the orthogonal group synchronization problem. It is desirable to study the problem with incomplete measurements which is seldom considered before, especially for the convergence analysis of various algorithms, since in real-world applications it is costly and even impractical to obtain complete measurements \cite{singer2011angular}. On the other hand, the error bound property which is independent from algorithms is powerful for analyzing various algorithms' convergence rate behavior by choosing suitable residual functions. 

    \item  Third, by taking the residual function related to fixed points of the GPM, the linear convergence result of the GPM to the global maximizers of problem \eqref{Problem} with general noise model under incomplete measurements is given with the help of the local error bound property. It is well worth highlighting that the conditions in our theoretical convergence are pure deterministic, hence it can cover certain cases with adversarial noise. We also specialize the result to the Erd\"os-R\'enyi measurement graph $\mathcal{G}([n],p)$ and the Gaussian case with the noise strength $\sigma =O (\frac{n^{1/4}p^{1/2}}{d}) $. 
\end{itemize}
	In addition, even though there is seemingly similarity between the orthogonal group $\od$ synchronization and the special orthogonal group $\sod$ synchronization problems, we point out the differences between them with/without the generative model, and explain the reason why it is more restrictive when applying the same analytical framework to $\sod$ synchronization problem.

The remainder of this paper is organized as follows. In section 2, useful notations and preliminary results are presented. The characterizations and properties of problem \eqref{Problem} without the generative model are provided in section 3. In section 4, we show more results of problem \eqref{Problem} based on the generative model, especially the local error bound property, from which we are able to prove the linear convergence result of the GPM with a general noise model, and also specialize it to the Erd\"os-R\'enyi measurement graph with the Gaussian noise setting in section 5. We end with some conclusions and future directions in section 6.



\section{Preliminaries}

Throughout the paper, we use the standard notations. Let the Euclidean space of all $m\times n$ real matrices $\mathbb{R}^{m\times n}$ be equipped with inner product $\langle X,Y\rangle:={\rm tr}(X^\top Y)$ for any $X,Y\in\mathbb{R}^{m\times n}$ and its induced norm $\Vert\cdot\Vert_F$, i.e., the Frobenius norm (denoted as $\Vert\cdot\Vert_2$ when reduced to vectors). Let $\mathbb{D}^{n}$, $\mathbb{S}^n$ be the space of $n\times n$ diagonal matrices and symmetric matrices, respectively. Let $\Vert \cdot\Vert_{*}$ and $\Vert \cdot\Vert$ be the nuclear norm and operator norm, and let $\left\| X \right\|_{\infty}: =  \max_{i\in[n]} \left\| X_i \right\|$ for a real matrix $X\in\mathbb{R}^{nd\times d}$, where $X_i$ is its $n$-th $d\times d$ block. Denote the positive semidefinite square root of a positive semidefinite matrix $X\in\mathbb{R}^{n\times n}$ by $X^{1/2}$, and also denote $X_i\in \mathbb{R}^{nd \times d}$ (resp. $X_{ij}\in \mathbb{R}^{d \times d}$) as the $i$-th block column (resp. $ij$-th block) of the matrix $X\in\mathbb{R}^{nd\times nd}$ formulated by $n\times n$ blocks of $d\times d$ real matrices. Denote $1_n\in \mathbb{R}^{n}$ as a $n$-dimension column vector with all entries equal to one.

The constraints in problem \eqref{Problem} involve the orthogonal group $\od$ which  is an embedded submanifold of $\mathbb{R}^{d\times d}$\cite[Section 3.3.2]{absil2009optimization}. The tangent space to $\odn$ at point $G$ is denoted by $\mathrm{T}_{G} \odn$. We consider the Riemannian metric on $\odn$ that is induced from the Euclidean inner product; i.e, for any $\xi, \eta \in \mathrm{T}_{G} \odn$, we have $\langle\xi, \eta\rangle_{G}:=\langle\xi, \eta\rangle=\operatorname{tr}\left(\xi^{\top} \eta\right)$. The Euclidean gradient, the Riemannian gradient and the Riemannian Hessian of a smooth function $f$ are denoted as $\nabla f$, $\operatorname{grad} f$ and $\operatorname{Hess} f$, respectively. 


Consider an arbitrary matrix $Z\in\mathbb{R}^{m\times n}$. Denote its singular values as $\sigma(Z):=(\sigma_1(Z),\ldots,\sigma_q(Z))\in \mathbb{R}_{+}^{q}$, where $q:=\min\{m,n\}$ and $\sigma_l(Z)$ is the $l$-th largest singular value. Let ${\rm Diag}(x)\in\mathbb{R}^{m\times n}$ (resp. ${\rm Diag}(X)\in\mathbb{R}^{nd\times nd}$) be a diagonal matrix (resp. block diagonal matrix) with diagonal elements being the elements of the vector $x\in\mathbb{R}^{q}$ (resp. the matrix $X\in\mathbb{R}^{nd\times d}$). We further define
$$\Xi(Z):=\Big\{(U,V)\in \mathcal{O}(m)\times\mathcal{O}(n)  \ |\ X=U {\rm Diag}(\sigma(Z))V^\top \Big\}$$
as the set of all pairs of orthogonal matrices satisfying the singular value decomposition of $Z$, and $\Xi(Z)$ can be decomposed into 
$$\Xi(Z)_1:=\Big\{U\in \mathcal{O}(m)  \ |\ (U,V)\in \Xi(Z)\  \text{for some}\ V\in\mathcal{O}(n) \Big\},$$ 
$$\Xi(Z)_2:=\Big\{V\in \mathcal{O}(m)  \ |\ (U,V)\in \Xi(Z)\  \text{for some}\ U\in\mathcal{O}(m) \Big\}.$$

Consider a set  $\mathcal{Z}\subseteq\mathbb{R}^{m\times n}$. The distance function to $\mathcal{Z}$ and the projection onto $\mathcal{Z}$ are denoted by 
${\rm dist}(\cdot,\mathcal{\mathcal{Z}})$ and $\Pi_{\mathcal{Z}}(\cdot)$, respectively, and they are defined by

  \begin{displaymath}
{\rm dist}(X,\mathcal{Z}):=\inf \limits_{Y\in \mathcal{Z}}\Vert X-Y\Vert_F
  \end{displaymath}
and
  \begin{displaymath}
  \Pi_{\mathcal{Z}}(X):=\{\bar{X}\in \mathcal{Z} \mid \Vert X-\bar{X}\Vert_F={\rm dist}(X,\mathcal{Z})\} \quad \text{for each}\ X\in\mathbb{R}^{m\times n},
  \end{displaymath}
respectively, where we adopt the convention that ${\rm dist}(x,\emptyset)=+\infty$.   

Next, we provide some useful results concerning matrix inequalities and probability.
\begin{fact}[Lipschitz continuity of matrix square root {\cite[Corollary VII.5.6; Theorem VII.5.7]{bhatia2013matrix}, \cite[Theorem 4.2]{bhatia2010modulus}}] \label{lemma: lipmatmodu}
	Let $X, Y \in \mathbb{R}^{d \times d}$ be two arbitrary matrices, then the following inequality holds:
	\begin{align*}
		& \| (XX^\top)^{1/2} - (YY^\top)^{1/2} \|_F \leq \sqrt{2} \| X - Y\|_F.
	\end{align*}
	Moreover, if $X,Y$ are both normal matrices (i.e., $XX^\top = X^\top X$, $YY^\top = Y^\top Y$), then 
	\begin{align*}
	 \| (XX^\top)^{1/2} - (YY^\top)^{1/2} \|_F \leq \| X - Y\|_F.
	\end{align*}
\end{fact}


\begin{fact}[Weyl's inequality {\cite[Theorem 4.3.1]{horn2012matrix}}] \label{lemma: weyl}
Let $X, Y \in \mathbb{S}^{d \times d}$ be two arbitrary matrices. Then for each $i\in[d]$, it follows that
$$\lambda_i(X+Y)\leq \lambda_{i-j+1}(X)+\lambda_{j}(Y), \quad \text{for each}\ j=1,\ldots,i,$$
where $\lambda_i(X)$ is the i-th largest eigenvalue of $X$. In particular,
$$\lambda_{\min}(Y)\ge \lambda_{\min}(X+Y)-\Vert X\Vert.$$
\end{fact}

\begin{fact}[{\cite[Fact 1]{so2009performance}}] \label{fact: epsilonnet}
	Let $\varepsilon>0$ and $N$ be an $\varepsilon$-net of the set $\{x\in\mathbb{R}^d \mid  \Vert x\Vert_2=1\}$ (i.e., for each $p \in \{x\in\mathbb{R}^d \mid  \Vert x\Vert_2=1\}$, there exists an $p^{\prime} \in N\subset\{x\in\mathbb{R}^d \mid  \Vert x\Vert_2=1\}$ such that $\left\|p-p^{\prime}\right\|_{2} \leq \varepsilon$). Then, for any $X\in\mathbb{R}^{d\times d}$, we have 
	\begin{align*}
		\| X\| \leq (1-\varepsilon)^{-2} \sup_{u \in N, v \in N} |u^\top X v|.
	\end{align*}
\end{fact}

\begin{lemma}[{\cite[Corollary 3.11]{bandeira2016sharp}}] \label{lemma: aop}
	Let $X$ be the $n \times n$ symmetric random matrix with entries $i\ge j$ satisfying $X_{ij} =b_{ij}\cdot g_{ij}$, where $g_{ij}$ are independent
	standard Gaussian random variables and $b_{ij}$ are given scalars. Moreover, we assume $\max\limits_{i,j}|b_{ij}|\leq 1$. Then 
	\begin{align*}
		P\left[\|X \| \geq (1+\varepsilon)\cdot\left(\max_{i \in [n]}\sqrt{\sum_{j\in [n]}b_{ij}^2}+ \max_{j\in [n]}\sqrt{\sum_{i \in [n]}b_{ij}^2} + \frac{5\sqrt{\log n}}{\sqrt{\log(1+\varepsilon)}} \right) + t \right] \leq \exp(-t^2/2),
	\end{align*}
for any $0<\varepsilon \leq 1/2$ and $t \geq 0$.
\end{lemma}

 
 
 
 





\section{Characterizations of Problem \eqref{Problem}}

In this section, we study the properties of problem \eqref{Problem} and all conclusions hold for general symmetric matrix $C$ (without loss of generality, $C_{ii}\succeq\bm{0}$ for all $i\in[n]$) without the generative model given by \eqref{eq: gsmodel3}.
\subsection{Basic facts and optimality conditions}


We start with some basic yet useful results about orthogonal group as well as optimality conditions of problem \eqref{Problem}. First, we denote a new distance function based on the quotient manifold which is useful in the analysis of the following parts as ${\rm d}(\cdot, \cdot)$, and it is defined by
\begin{equation*}
	{\rm d}(X,Y) :=  \min\limits_{g\in \od} \Vert X-Yg\Vert_F,\quad \text{for each}\  X,Y\in\mathbb{R}^{nd\times d}.
\end{equation*}
Moreover, if $X,Y \in \odn$, then 
$${\rm d}^2(X,Y) = 2(nd - \max_{g\in \od } \langle Y^\top X,g \rangle) = 2(nd - \| Y^\top X\|_*).$$
The projection 
onto the orthogonal group $\od$, denoted as $\Pi_{\od}(\cdot)$, has a closed-form solution that for each $Z\in \mathbb{R}^{d\times d}$ with singular value decomposition $Z = U_Z\Sigma_ZV_Z^\top$, 
\begin{align} 
	& \Pi_{\mathcal{O}(d)}(Z)=\mathop{\arg\min}_{X\in \mathcal{O}(d)} \Vert X-Z\Vert_F= U_{Z} V_{Z}^\top.  \label{eq: orthoprocru}
\end{align}
Note that the solution to \eqref{eq: orthoprocru} need not be unique as the polar decomposition of $Z$ may not be unique.
Moreover, we define $ \Pi_{\odn}(G) := [\Pi_{\od}(G_1);\ldots;\Pi_{\od}(G_n)]$, where $G=[G_1;\ldots;G_n]\in\mathbb{R}^{nd\times d}$. Define a linear operator symblockdiag: $\mathbb{R}^{nd\times nd}\rightarrow \mathbb{S}^{nd}$ as
\begin{equation}
  {\rm symblockdiag}(X)_{ij}=\left\{
  \begin{aligned}
  &\frac{X_{ii}+X_{ii}^\top}{2}, & \ &\textrm{ if }\ i=j,\\
  & \bm{0}, & \ &\textrm{ otherwise},
  \end{aligned}
  \right.
  \end{equation}
which symmetrizes diagonal blocks of $X$ and sets all
other blocks to be $\bm{0}$. Based on this definition, it follows that
\begin{align*}
\odn &=\{G\in\mathbb{R}^{nd\times d} \mid G=[G_1;\ldots;G_n],\ G_i G_i^\top = G_i^\top G_i=I_d,\  i\in [n]\} \\
&=\{G\in\mathbb{R}^{nd\times d} \mid \text{symblockdiag}(GG^\top)=I_{nd}\}.
\end{align*}
The tangent space of $\odn $ \cite[Example 3.5.3]{absil2009optimization} at each $G$ is 
\begin{align*}
\mathrm{T}_G \odn
&=\{H = [H_1;\dots;H_n]\in \mathbb{R}^{nd \times d} \mid H_i = E_iG_i,~ E_i = -E_i^\top,~ i\in [n]\}\\
&=\{H\in\mathbb{R}^{nd\times d}\mid  \text{symblockdiag}(HG^\top )=\bm{0}\}.
\end{align*}
That is, if $H \in \mathbb{R}^{nd \times d}$ is in the tangent space $\mathrm{T}_G \odn$, $H_i G_i^\top$ is a skew-symmetric matrix for any $i \in [n]$. The orthogonal projection of $ X \in \mathbb{R}^{nd\times d}$ to the tangent space to $\odn$ at $G$ is given by
$$\Pi_{\mathrm{T}_G \odn}(X)=X-\text{symblockdiag}(XG^\top)G.$$
Let $S(G):=\text{symblockdiag}(C G G^\top)-C $. Note that by our choice of the Riemannian metric, we have the Riemannian gradient of $f$ in problem \eqref{Problem} at $G$ as the orthogonal projection of $\nabla f(G)$ onto the tangent space, i.e.,
\begin{equation*}
	\text{grad} f(G)=\Pi_{\mathrm{T}_{G}\odn}(\nabla f(G))=2\left(C-\text{symblockdiag}(C G G^\top )\right) G=-2S(G)G.
\end{equation*}
More specifically, for each $i\in [n]$, we have $[\text{grad} f(G)]_i =C_{i}^\top G-G_i G^\top C_{i} G_i$.
Furthermore, the Riemannian Hessian of $f$ at $G$ along each direction $H\in \mathrm{T}_{G} \odn$ is given by
\begin{align}
\text{Hess} f(G)[H]
&=\Pi_{\mathrm{T}_{G}\odn} \left(\mathrm{D}\operatorname{grad} f(G) [H] \right)\notag\\
&=-2\Pi_{\mathrm{T}_{G}\odn}  \left(\text{symblockdiag} (C(GH^\top+HG^\top))G+S(G)H\right)\notag\\
&=-2\Pi_{\mathrm{T}_{G}\odn} \left(S(G)H\right),\notag
\end{align}
where $\mathrm{D} \operatorname{grad} f(G)[H]$ is the classical directional derivative. With the help of the Riemannian gradient and Hessian, we have the following definitions about the first- and second-order critical point of the problem \eqref{Problem}.
\begin{defi}\label{def-crit}
Consider the problem \eqref{Problem}. A point $G\in \odn$ is said to be a\\ 
{\rm (i)} first-order critical point if $\text{grad} f(G)=\bm{0}$, i.e., $S(G)G=\bm{0}$;\\
{\rm (ii)} second-order critical point if $\text{grad} f(G)=\bm{0}$, and $\langle H,\text{Hess} f(G)H\rangle \leq 0 $ for all $H\in \mathrm{T}_{G} \odn$, i.e.,
\begin{equation*}
	S(G)G=\bm{0},\quad \langle H,  S(G)H\rangle\ge0.
\end{equation*}
\end{defi}

\begin{fact}[{\cite[Lemma 3.2]{boumal2015riemannian}}]\label{boumal-fact-2015}
	All the local maximizers of problem \eqref{Problem} are second-order critical points of problem \eqref{Problem}.
	\end{fact}




\subsection{Fixed points, global maximizers and critical points}
For characterizing problem \eqref{Problem}, we introduce fixed points of the projected gradient ascent method. Let $\alpha\ge0$, $G\in\odn$ and define the function $f_{G,\alpha}:\odn\rightarrow \mathbb{R}$ by  
\begin{align*}
f_{G,\alpha}(G')
:= -f(G)-\langle \nabla f(G),G'-G\rangle+\alpha\Vert G'-G\Vert_F^2\quad \text{for each}\ G'\in \odn.
\end{align*}
Then, with $\alpha\ge0$ we define the multi-valued mapping $\mathcal{T}_\alpha: \odn \rightrightarrows \odn$ for each $G\in \odn$ by
\begin{align}\label{def-T0}
\mathcal{T}_\alpha(G):=\mathop{\arg\min}_{G' \in \odn}f_{G,\alpha}(G').
\end{align}
In particular, when $\alpha>0$ it follows that
\begin{align}\label{def-T1}
\mathcal{T}_\alpha(G)=\mathop{\arg\min}\limits_{G'\in \odn }\left\{\left\| G'-\left(G+\frac{1}{2\alpha}\nabla f(G)\right)\right\|_F^2 \right\}
=\Pi_{\odn}(\tilde{C}G),
\end{align}
where $\tilde{C}:=C+\alpha I_{nd}$, and when $\alpha=0$ it reduces to
\begin{align} \label{def-T2}
\mathcal{T}_0(G)=\mathop{\arg\min}\limits_{G'\in \odn}\left\{\Vert G'-\nabla f(G)\Vert_F^2 \right\}=\Pi_{\odn}(CG).
\end{align}
Then, we can unify the above two cases with
\begin{align}\label{def-T}
	\mathcal{T}_\alpha(G)= \Pi_{\odn}(\tilde{C}G), \quad \alpha \geq 0.
\end{align}

In what follows, we investigate the fixed points (FPs) of $\mathcal{T}_\alpha$ (i.e., $G\in  \mathcal{T}_\alpha(G)$ for any $G\in\odn$), global maximizers (GMs), first-order critical points (FOCPs) and second-order critical points (SOCPs) of problem \eqref{Problem},  with the following relationships:
 
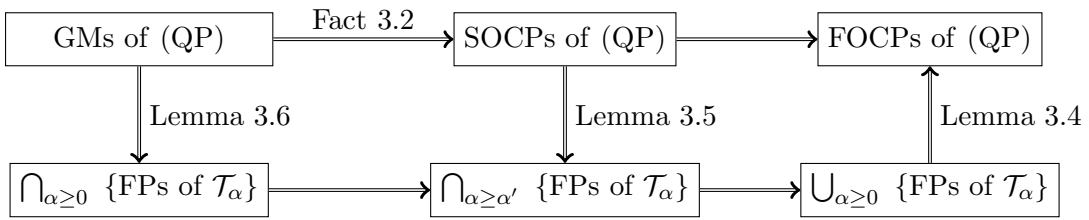
\begin{center}
    \begin{tikzpicture}[scale=0.8]
      \node[rectangle,
      minimum width =100pt ,
      minimum height =20pt ,draw=black] (1) at(-6,4){GMs of \eqref{Problem}};
      \node[rectangle,
      minimum width =80pt ,
      minimum height =20pt ,draw=black] (2) at(-6,1.5){$\bigcap_{\alpha\ge0}\ \left\{\text{FPs of $\mathcal{T}_\alpha$}\right\}$};
	  \node[rectangle,
      minimum width =20pt ,
      minimum height =20pt ,draw=black] (3) at(1,4){SOCPs of \eqref{Problem}}; 
	  \node[rectangle,
      minimum width =20pt ,
      minimum height =20pt ,draw=black] (4) at(1,1.5){$\bigcap_{\alpha\ge\alpha'}\ \left\{\text{FPs of $\mathcal{T}_\alpha$}\right\}$};
	  \node[rectangle,
      minimum width =20pt ,
      minimum height =20pt ,draw=black] (5) at(7,1.5){$\bigcup_{\alpha\ge0}\ \left\{\text{FPs of $\mathcal{T}_\alpha$}\right\}$};      
	  \node[rectangle,
      minimum width =20pt ,
      minimum height =20pt ,draw=black] (6) at(7,4){FOCPs of \eqref{Problem}};
      \draw[double, ->](1) --(3) node[midway, sloped,above] {Fact \ref{boumal-fact-2015}} ;      
	  \draw[double, ->] (1) --(2) node[midway, right] {Lemma \ref{prop: globalfix}};
      \draw[double, ->] (3) --(4) node[midway, right] {Lemma \ref{lemma: socpsum} } ;
      \draw[double, ->] (5) --(6) node[midway, right] {Lemma \ref{lemma-equi2}}; 
      \draw[double, ->] (2) --(4);
	  \draw[double, ->] (3) --(6);
	  \draw[double, ->] (4) --(5);
      \end{tikzpicture}
    \end{center}
where $\alpha'\ge0$ is sufficiently large to guarantee that for all $i\in[n]$, $C_i^\top GG_i^\top + \alpha' I_d\succeq\bm{0}$ with $G$ being a second-order critical point of problem \eqref{Problem}.

\begin{lemma}[Characterization of fixed points]
	\label{lemma-equi}
Let $\alpha\ge0$ and $G\in\odn$. Then the following assertions are equivalent:\\
{\rm (a)} $G$ is a fixed point of $ \mathcal{T}_\alpha$;\\
{\rm (b)} For each $i\in [n]$, ${\tilde{C}_i^\top} G G_i^\top=U_{{\tilde{C}_i^\top}  G}\Sigma_{{\tilde{C}_i^\top} G}U^\top_{{\tilde{C}_i^\top} G}$, where $U_{{\tilde{C}_i^\top} G}\in \Xi({\tilde{C}_i^\top} G)_1$;\\
{\rm (c)} For each $i \in [n]$, ${\rm tr}({\tilde{C}_i^\top} G G_i^\top)=\Vert {\tilde{C}_i^\top} G\Vert_*$;\\
{\rm (d)} ${\rm tr}(G^\top \tilde{C} G)=\sum\limits_{i=1}^n \Vert {\tilde{C}_i^\top} G\Vert_*$.
\end{lemma}
\begin{proof}
Note that ${\rm (a)}\Rightarrow{\rm (b)}\Rightarrow{\rm (c)}\Rightarrow{\rm (d)}$ is evident (by definition), then it suffices to show ${\rm (d)}\Rightarrow {\rm (a)}$.
From ${\rm (d)}$ we know that ${\rm tr}(G^\top \tilde{C} G)=\sum\limits_{i=1}^n \Vert {\tilde{C}_i^\top} G\Vert_*$, and this together with 
$$\left\langle {\tilde{C}_i^\top} G,G_i\right\rangle\leq \Vert {\tilde{C}_i^\top} G\Vert_* \Vert G_i\Vert = \Vert {\tilde{C}_i^\top} G\Vert_*$$
(by the fact that the nuclear norm is the dual norm of the operator norm) implies that $\langle {\tilde{C}_i^\top} G,G_i\rangle= \Vert {\tilde{C}_i^\top} G\Vert_*$ for each $i\in [n]$. Then we have
\begin{align}\label{C-fix-key}
	\left\langle\Sigma_{{\tilde{C}_i^\top}G},U_{{\tilde{C}_i^\top}G}^\top G_i V_{{\tilde{C}_i^\top}G}-I_d\right\rangle=\left\langle\Sigma_{{\tilde{C}_i^\top}G}^{++},U_{{\tilde{C}_i^\top}G}^{1\top} G_i V_{{\tilde{C}_i^\top}G}^1-I_r\right\rangle= 0, \quad \text{for each}\ i\in[n],
\end{align} 
where $r={\rm rank}({\tilde{C}_i^\top}G)$, $\Sigma_{{\tilde{C}_i^\top}G}^{++}\in\mathbb{D}_{++}^{r}$ contains all positive singular values and $U_{{\tilde{C}_i^\top}G}^{1}\in\mathbb{R}^{d\times r}$ (resp. $V_{{\tilde{C}_i^\top}G}^{1}\in\mathbb{R}^{d\times r}$) is the matrix of first $r$ columns of $U_{{\tilde{C}_i^\top}G}$ (resp. $V_{{\tilde{C}_i^\top}G}$).
Thus, from \eqref{C-fix-key} and $\left( U_{{\tilde{C}_i^\top}G}^{1\top} G_i V_{{\tilde{C}_i^\top}G}^1\right)_{jj}\leq 1$ for any $j\in[r]$, one has that $G_i V_{{\tilde{C}_i^\top}G}^{1} =U_{{\tilde{C}_i^\top}G}^1$, which together with $G_i\in\mathcal{O}(d)$ implies that there exist $U=[U_{{\tilde{C}_i^\top}G}^1\ U^0]\in\mathcal{O}(d)$ and $V=[V_{{\tilde{C}_i^\top}G}^1\ V^0]\in\mathcal{O}(d)$ such that $G_i=UV^\top$. Since it can be directly verified that $(U,V)\in \Xi({\tilde{C}_i^\top}G)$, we know that $G_i\in\mathcal{T}_\alpha(G)_i$ for each $i\in[n]$, and consequently $G\in \mathcal{T}_\alpha(G)$. The proof is complete.
\end{proof}

By considering the optimality conditions of problem \eqref{Problem}, 
it is important to investigate the relationships between the fixed points with the first-order and second-order critical points, even with the global maximizers of problem \eqref{Problem}. 

\begin{lemma}[Relationship between fixed points and critical points]
\label{lemma-equi2} 
Let $\alpha\ge0$ and $G \in \odn$. Then the following statements hold:\\
{\rm (a)} $G$ is a first-order critical point of problem \eqref{Problem} if and only if ${\tilde{C}_i^\top} G G_i^\top\in \mathbb{S}^d$.\\ 
{\rm (b)} If $G$ is a fixed point of $\mathcal{T}_\alpha$, then $G$ is a first-order critical point of problem \eqref{Problem}.\\ 
{\rm (c)} If $G$ is a first-order critical point of problem \eqref{Problem} with ${\tilde{C}_i^\top} GG_i^\top \succeq \bm{0}$ for each $i \in [n]$, then $G$ is a fixed point of $\mathcal{T}_\alpha$.
\end{lemma}
\begin{proof}
From Definition \ref{def-crit}(a) we know that $S(G)G=\bm{0}$ if $G$ is a first-order critical point of problem \eqref{Problem}. Then the statement (a) follows from the definition of $S(G)$:
	\begin{align*}
		(S(G)G)_i=\bm{0} \iff G_i G^\top C_iG_i = C_i^\top G \iff {C_i^\top G}G_i^\top = G_iG^\top C_i \iff {\tilde{C}_i^\top G}G_i^\top = G_iG^\top \tilde{C}_i,
	\end{align*}
where the last equivalence is due to the fact that $ {C_i^\top G}G_i^\top=(\tilde{C}-\alpha I_{nd})_i^\top G G_i^\top={\tilde{C}_i^\top} G G_i^\top-\alpha I_d$.
For statement (b), let $G\in \mathcal{T}_\alpha(G)$ be a fixed point. Then from Lemma \ref{lemma-equi}(b) we know ${\tilde{C}_i^\top} G G_i^\top$ is symmetric for $i \in [n]$, which implies that $G$ is a first-order critical point of problem \eqref{Problem}. Consider the statement (c), if $G$ is a first-order critical point, then $\tilde{C}_i^\top G G_i^\top$ is symmetric. Then by the assumption ${\tilde{C}_i^\top} G G_i^\top\succeq \bm{0}$, for each $i\in [n]$, it follows that
\begin{align*}
	\Vert {\tilde{C}_i^\top} G\Vert_* &={\rm tr}\left( \left( {\tilde{C}_i^\top} G ({\tilde{C}_i^\top} G)^\top \right)^{1/2} \right)={\rm tr}\left(\left( {\tilde{C}_i^\top} G G_i^\top G_i ({\tilde{C}_i^\top G)^\top} \right)^{1/2} \right) \\
	&  ={\rm tr}\left( \left( ({\tilde{C}_i^\top} G G_i^\top)^2 \right)^{1/2} \right) ={\rm tr}\left({\tilde{C}_i^\top} G G_i^\top \right).
\end{align*}
At this point, we know from Lemma \ref{lemma-equi}(c) that $G$ is a fixed point of $\mathcal{T}_\alpha$. The proof is complete.
\end{proof}
The following lemma about the second-order critical points has been mentioned in \cite{ling2020solving}. For the sake of completeness, we present the proof. 
\begin{lemma}[Properties of second-order critical points] \label{lemma: socpsum}
	Suppose that $G \in \odn$ is a second-order critical point of problem \eqref{Problem}. Then for each $i\in [n]$ the sum of the smallest two eigenvalues of $C_i^\top GG_i^\top -C_{ii}$ is nonnegative.
\end{lemma}
\begin{proof}
Let $H= [H_1;\dots;H_n]\in T_G \odn$ be such that $H_i = E_iG_i$ and $E_i = -E_i^\top$ for all $i\in [n]$.
Since $G \in \odn$ is a second-order critical point, one has from the definition for each $i\in [n]$ that $\langle C_i^\top GG_i^\top -C_{ii},  H_iH_i^\top  \rangle \geq 0$, which implies that $$\langle  C_i^\top G G_i^\top-  C_{ii}, E_iE_i^\top  \rangle  \geq 0.$$
From Lemma \ref{lemma-equi2}(a) we know that ${C_i^\top} G G_i^\top\in \mathbb{S}^d$, and then $C_i^\top G G_i^\top-  C_{ii}\in \mathbb{S}^d$. Thus, let the spectral decomposition of $C_i^\top G G_i^\top-  C_{ii}$ be $U\Lambda U^\top$, where $\Lambda:= \text{Diag}(\lambda_1,\dots,\lambda_d)$ with $\lambda_1 \geq \dots\geq \lambda_d$, and we obtain
	\begin{equation}\label{soc-key}
		\langle U \Lambda U^\top, E_iE_i^\top  \rangle  = \langle  \Lambda, U^\top E_iU U^\top E_i^\top U  \rangle  \geq 0.
	\end{equation}
	Let $E_i = U\hat{E}_iU^\top$ with $\hat{E}_i \in \mathbb{R}^{d \times d}$ being a zero matrix except for $(\hat{E}_i)_{d,d-1}=1,(\hat{E}_i)_{d-1,d}=-1 $, and then it follows from \eqref{soc-key} that $\lambda_{d-1}+ \lambda_d \geq 0$. The proof is complete.
\end{proof}

Based on the results in Lemma \ref{lemma: socpsum}, for a second-order critical point $G\in\odn$ it follows that $\lambda_{d-1}(C_i^\top GG_i^\top -C_{ii})+\lambda_{d}(C_i^\top GG_i^\top -C_{ii})\geq0$ holds for all $i\in[n]$. If we take $\alpha$ be sufficiently large (e.g., $\alpha \ge \max\{0, \max_{i\in[n]}\{-\lambda_{\min}(C_i^\top GG_i^\top)\}\}$), then $C_i^\top GG_i^\top + \alpha I_d\succeq\bm{0}$ for all $i\in[n]$, and consequently by Lemma \ref{lemma-equi2}(c), the second-order critical point becomes a fixed point of $\mathcal{T}_\alpha$ with such given $\alpha$. However, the statement may not necessarily hold for arbitrary $\alpha$ (including the case $\alpha = 0$). Nevertheless, we still have the following lemma showing that any global maximizer is a fixed point as long as $\alpha \geq 0$.


\begin{lemma}[Global maximizers are fixed points] \label{prop: globalfix}
Let	$\hat{G}\in\odn$ be a global maximizer of problem \eqref{Problem}. Then $\hat{G}$ is a fixed point of $\mathcal{T}_\alpha$ for any $\alpha\ge0$. Moreover, $C_i^\top \hat{G}\hat{G}_i^\top - C_{ii} \succeq \bm{0}$ holds for all $i \in [n]$.
\end{lemma}
\begin{proof}
	In order to derive the conclusion, we will show that $C_i^\top \hat{G} \hat{G}_i^\top -C_{ii}$ is positive semidefinite for all $i\in[n]$. Granting this, one has that $C_i^\top \hat{G} \hat{G}_i^\top$ is positive semidefinite for each $i \in [n]$ since $C_{ii}\succeq\bm{0}$. Consequently, we know from Lemma \ref{lemma-equi2}(c) that $\hat{G}$ is a fixed point.
	
	
To proceed, note that for each $i\in[n]$, $C_i^\top \hat{G}\hat{G}_i^\top - C_{ii}$ is symmetric (from the property of first-order critical point by Lemma \ref{lemma-equi2} (a)), then we have the spectral decomposition 
$$C_i^\top \hat{G}\hat{G}_i^\top- C_{ii} = \left(\sum_{j\neq i} C_{ij}\hat{G}_j\right)\hat{G}_i^\top  = U_i \text{Diag}(\lambda_{i,1},\ldots,\lambda_{i,d})U_i^\top,$$ 
where $\lambda_{i,1}\ge \cdots\ge\lambda_{i,d}$ are eigenvalues and $U_i\in\mathbb{R}^{d\times d}$ is the eigenvector matrix of $C_i^\top \hat{G}\hat{G}_i^\top - C_{ii}$. Suppose to the contrary that there exists an index $t \in [n]$ such that $C_t^\top \hat{G}\hat{G}_t^\top - C_{tt}$ is not positive semidefinite, then from Lemma \ref{lemma: socpsum} we know that the sum of the smallest two eigenvalues of  $C_t^\top \hat{G}\hat{G}_t^\top -C_{tt}$  is nonnegative (i.e., $\lambda_{t,d-1}+\lambda_{t,d}\ge0$), which implies that $\lambda_{t,d}<0$.
Let
	\begin{align*}
			\bar{G} := \left(\hat{G}_1;\ldots;\hat{G}_{t-1};\Pi_{\mathcal{O}_d}\left(\sum_{j\neq t} C_{tj}\hat{G}_j\right);\hat{G}_{t+1};\ldots;\hat{G}_n \right)\in \odn.
	\end{align*}
Since for any $X\in\mathbb{R}^{d\times d}$, we have $\Pi_{\od}(X)=U_X V_X^\top$ from \eqref{eq: orthoprocru}, and then 
$$\left\langle X, \Pi_{\od}(X)\right\rangle={\rm tr}\left(U_X \Sigma_X V_X^\top V_X U_X^\top\right)=\Vert X\Vert_*,$$
which implies that
\begin{align*}
	\sum_{j\neq t} \textrm{tr}\left(\bar{G}_t^\top C_{tj} \hat{G}_j\right)=  \left\langle \sum_{j\neq t} C_{tj}\hat{G}_j, \Pi_{\mathcal{O}_d}\left(\sum_{j\neq t} C_{tj}\hat{G}_j\right) \right\rangle=  \bigg\Vert \sum_{j\neq t} C_{tj}\hat{G}_j\bigg\Vert_*
	=  \Vert C_{t}\hat{G}\hat{G}_t^\top-C_{tt}\Vert_*.
\end{align*}
Then it follows that
\begin{align}\label{3.6-key1}
	\sum_{j\neq t} \textrm{tr}\left((\bar{G}_t-\hat{G}_t)^\top C_{tj} \hat{G}_j\right)
	&=\sum_{j\neq t} \textrm{tr}\left(\bar{G}_t^\top C_{tj} \hat{G}_j\right)-\sum_{j\neq t} \textrm{tr}\left(\hat{G}_t^\top C_{tj} \hat{G}_j\right)\notag\\
	&=  \Vert C_{t}\hat{G}\hat{G}_t^\top-C_{tt}\Vert_*-\textrm{tr}\left(\hat{G}_t^\top  C_{t} \hat{G}-C_{tt}\right)\notag\\
	&=  \sum_{l=1}^d (\sigma_{t,l}-\lambda_{t,l}) = -2\lambda_{t,d},
\end{align}
where $\sigma_{t,l}$ is $l$-th largest singular value of $C_{t}\hat{G}\hat{G}_t^\top-C_{tt}$. Hence, we know from \eqref{3.6-key1} and $\lambda_{t,d}<0$ that

\begin{align*}
	&\textrm{tr}\left(\bar{G}^\top C \bar{G}\right) - \textrm{tr}\left(\hat{G}^\top C \hat{G}\right)\notag\\ 
	&=\sum_{i=1}^n\sum_{j=1}^n \textrm{tr}\left(\bar{G}_i^\top C_{ij} \bar{G}_j\right) - \sum_{i=1}^n\sum_{j=1}^n\textrm{tr}\left(\hat{G}_i^\top C_{ij} \hat{G}_j\right)\\
	&=\sum_{i=1}^n \textrm{tr}\left( \bar{G}_i^\top C_{it} \bar{G}_t \right)+\sum_{j=1}^n \textrm{tr}\left(\bar{G}_t^\top C_{tj} \bar{G}_j\right)- \sum_{i=1}^n \textrm{tr}\left( \hat{G}_i^\top C_{it} \hat{G}_t \right)-\sum_{j=1}^n \textrm{tr}\left(\hat{G}_t^\top C_{tj} \hat{G}_j\right)\notag\\
	&=\sum_{i\neq t} \textrm{tr}\left( \hat{G}_i^\top C_{it} (\bar{G}_t-\hat{G}_t) \right)+\sum_{j\neq t} \textrm{tr}\left((\bar{G}_t-\hat{G}_t)^\top C_{tj} \hat{G}_j\right)\notag\\
	&=2\sum_{j\neq t} \textrm{tr}\left((\bar{G}_t-\hat{G}_t)^\top C_{tj} \hat{G}_j\right)
	 = -4\lambda_{t,d} >0,
\end{align*}
which contradicts with the fact that $\hat{G}$ is a global maximizer. The proof is complete.
\end{proof}



It is a natural question to see whether the results in this section are also available for problem \eqref{Problem} with the special orthogonal group constraints $\sodn$, i.e.,
\begin{equation}\label{Problem-S}\tag{QP-S}
	\begin{array}{ll}
	\max\limits_{G\in \sodn } f(G):=\textrm{tr}(G^\top C G).\\
	\end{array}
	\end{equation}
The projection onto $\sod$ denoted as $\Pi_{\sod}(\cdot)$ has a closed-form solution \cite[Lemma 4]{liu2020unified} that for any $X\in\mathbb{R}^{d\times d}$,
$$\Pi_{\mathcal{S O}(d)}(X)=U_{X} I_{X} V_{X}^{\top},\quad \text{where}\
I_{X}:={\rm Diag}\left([1;\ldots;1;\operatorname{det}(U_{X} V_{X}^{\top})]\right) \in \mathbb{R}^{d \times d}.
$$
Let $\alpha\ge0$ and we have the following characterization for fixed points of the multi-valued mapping $\mathcal{ST}_\alpha: \sodn \rightrightarrows \sodn$ defined by
\begin{align}\label{def-ST0}
\mathcal{ST}_\alpha(G):=\mathop{\arg\min}_{G' \in \sodn}f_{G,\alpha}(G')=\Pi_{\mathcal{S O}(d)}(\tilde{C}G),\quad \text{for any}\ G\in \sodn.
\end{align}

\begin{lemma}[Characterization for fixed points of $\mathcal{ST}_\alpha$]\label{lemma-sodequi2}
	Let $\alpha\ge0$ and $G\in\sodn$. The following assertions are equivalent:
\begin{align}
G\in  \mathcal{ST}_\alpha(G)
&\Longleftrightarrow {\tilde{C}_i^\top} G G_i^\top=U_{{\tilde{C}_i^\top}  G}\Sigma_{{\tilde{C}_i^\top} G} I_{{\tilde{C}_i^\top} G} U^\top_{{\tilde{C}_i^\top} G},\ \text{for each}\ i\in [n]\notag\\
&\Longleftrightarrow {\rm tr}({\tilde{C}_i^\top} G G_i^\top)=\Vert {\tilde{C}_i^\top} G\Vert_* - 2\cdot \asign{ \det(\tilde{C}_i^\top G)} \sigma_{i,d},\ \text{for each}\ i\in [n]\notag
\end{align}
where $\sigma_{i,d}$ is the smallest singular value of $\tilde{C}_i^\top G$, and $\asign{t} = 0$ if $t\geq 0$, $\asign{t} = 1$ otherwise.
\end{lemma}


\begin{lemma}[Second-order critical points are fixed points of $\mathcal{ST}_\alpha$]
	Let	$G\in\sodn$ be a second-order critical point of problem \eqref{Problem-S} which is also a second-order critical point of problem \eqref{Problem}. Then $G$ is a fixed point of $\mathcal{ST}_\alpha$ for any $\alpha\ge0$. 
\end{lemma}
\begin{proof}
From Lemma \ref{lemma: socpsum} we know that for each $i\in[n]$, the sum of the smallest two eigenvalues of $\tilde{C}_i^\top GG_i^\top$ is nonnegative. For the case $\lambda_{i,d-1}\ge\lambda_{i,d} \ge0$, since $\det(\tilde{C}_i^\top G)\ge0$  we know that 
$${\rm tr}({\tilde{C}_i^\top} G G_i^\top)=\Vert {\tilde{C}_i^\top} G\Vert_*=\Vert {\tilde{C}_i^\top} G\Vert_* - 2\cdot \asign{ \det(\tilde{C}_i^\top G)} \sigma_{i,d}.$$
On the other hand, suppose that $\lambda_{i,d-1} >0$, $\lambda_{i,d} <0$ and it follows that $\det(\tilde{C}_i^\top G)<0$. Then $${\rm tr}(\tilde{C}_i^\top GG_i^\top) = \| \tilde{C}_i^\top G\|_* - 2 \sigma_{i,d}=\Vert {\tilde{C}_i^\top} G\Vert_* - 2\cdot \asign{ \det(\tilde{C}_i^\top G)} \sigma_{i,d}.$$ 
Thus, from Lemma \ref{lemma-sodequi2} we know that $G$ is a fixed point of $\mathcal{ST}_\alpha$.
\end{proof}




Consequently, we have the following relationships for problem \eqref{Problem-S}:
\begin{align}
	\text{GMs of \eqref{Problem-S}}\ \mbox{\Large$\subseteq$}\ \text{SOCPs of \eqref{Problem-S}}\ \mbox{\Large$\subseteq$} \bigcap_{\alpha\ge 0} \left\{\text{FPs of $\mathcal{ST}_\alpha$}\right\} \mbox{\Large$\subseteq$} \bigcup_{\alpha\ge 0} \left\{\text{FPs of $\mathcal{ST}_\alpha$}\right\} \mbox{\Large$\subseteq$}\ \text{FOCPs of \eqref{Problem-S}}. \notag
	\end{align}
However, let $\hat{G}$ be a global maximizer of problem (QP-S), then the property that $C_i^\top \hat{G}\hat{G}_i^\top - C_{ii} \succeq \bm{0}$, $i \in [n]$ does not necessarily hold anymore, which means that the second part of Lemma \ref{prop: globalfix} fails. A counter example is as follows.
\begin{example}
Let $d = 3$, $n=2$ and  $C = \begin{pmatrix} \bm{0} & -I_3 \\ -I_3 & \bm{0} \end{pmatrix}$. Then problem \eqref{Problem-S} becomes 
\begin{equation*}
\max_{G_1, G_2 \in \mathcal{SO}(3)} -2\cdot{\rm tr}(G_1^\top G_2).
\end{equation*}
It can be verified that $\hat{G}=[\hat{G}_1;\hat{G}_2]$ with $\hat{G}_1 = I_3$ and $\hat{G}_2 =  \begin{pmatrix} 1 & 0 & 0  \\ 0 & -1 & 0 \\ 0 & 0 & -1  \end{pmatrix}$ is a global optimal solution. However, $C_2^\top \hat{G}\hat{G}_2^\top = \begin{pmatrix} -1 & 0 & 0  \\ 0 & 1 & 0 \\ 0 & 0 & 1  \end{pmatrix} $ is not positive semidefinite.
\end{example}


\section{Local error bound property}


In this section, we shall utilize the generative model given by problem \eqref{Problem} to obtain more powerful results. Among these, the most important and challenging one is the local error bound property of problem \eqref{Problem}, which is widely employed to derive the convergence rate analyses of various algorithms for different types of optimization problems  \cite{luo1993error,zhou2017unified,liu2019quadratic,zhang2020new,wang2021linear}. 

Our local error bound result provides a computable upper bound on the distance between any feasible point $G\in \odn$ around the ground truth $G^\star$ and the set of global maximizers of problem \eqref{Problem}. To begin with, we first show that any global maximizer of problem \eqref{Problem} is close to the ground truth $G^\star$. Without loss of generality, we assume $W_{ii}=\mu$ for all $i\in [n]$, where $\mu>0$ is a constant. 

\begin{lemma}[Distance between $\hat{G}$ and $G^\star$] \label{dist-GstarGhat}
Let $G \in \mathbb{R}^{nd \times d}$ be such that $G^\top G = nI_d$ and ${\rm tr}(G^{\top} C G) \geq {\rm tr} \left(G^{\star\top} C G^{\star} \right)$ (e.g., $G$ is a global maximizer of problem \eqref{Problem}). Then it follows that
  \begin{equation}\label{dist-GstarGhat-ineq}
  {\rm d}(G,G^\star)=\min\limits_{g\in \od } \Vert G-G^{\star}g\Vert_F\leq   \frac{4\sqrt{d} \left(   \left\| W-{\mu \cdot 1_n1_n^\top} \right\| + \left\| A \circ \Delta \right\|  \right) }{\sqrt{n}\mu}.
  \end{equation}
  \end{lemma}
\begin{proof}
Without loss of generality, we assume that $G$ satisfies $\textrm{tr}(G^{\top} G^{\star})=$ $\|G^{\top} G^{\star}\|_{*}$ and consequently ${\rm d}(G,G^\star) =  \Vert G-G^\star\Vert_F$.
We know that
$$\textrm{tr} \left(G^{\top} \left(A \circ G^{\star}G^{\star\top} \right) G\right)
=\mu \cdot \textrm{tr} \left( G^\top    G^{\star} G^{\star\top} G \right) + \textrm{tr} \left(G^{\top} \left((A- \mu\cdot 1_{nd}1_{nd}^\top) \circ G^{\star} G^{\star \top} \right) G \right),$$
and then it follows from $C = A \circ (G^\star G^{\star\top} + \Delta )$ that
\begin{align}\label{GCG-ineq}
\textrm{tr} \left(G^{\top} C G \right)
&=\textrm{tr} \left(G^{\top} \left(A \circ G^{\star}G^{\star\top} \right) G\right)+\textrm{tr}\left(G^{\top}(A \circ \Delta) G \right)\notag\\
&=\mu \cdot \textrm{tr} \left( G^\top    G^{\star} G^{\star\top} G \right) +\textrm{tr} \left(G^{\top} \left((A-\mu\cdot 1_{nd}1_{nd}^\top) \circ G^{\star} G^{\star\top} + (A \circ \Delta) \right) G \right).
\end{align}
Thus, from $\textrm{tr}(G^{\top} C G) \geq \textrm{tr} \left(G^{\star \top} C G^{\star} \right)$ and \eqref{GCG-ineq},
 one has that
\begin{align}\label{dist-GstarGhat-key1}
& \mu \left(n^{2} d-\textrm{tr}(G^{\top} G^{\star}G^{\star\top} G ) \right)  \nonumber\\
& \leq \textrm{tr} \left(G^{\top} \left((A-{\mu\cdot 1_{nd}1_{nd}^\top}) \circ G^{\star} G^{\star\top} + (A \circ \Delta) \right) G \right) - \textrm{tr} \left(G^{\star\top} \left((A-{\mu\cdot 1_{nd}1_{nd}^\top}) \circ G^{\star}G^{\star\top}+ (A \circ \Delta) \right) G^{\star} \right)  \nonumber \\
& = \textrm{tr} \left(\left(G-G^{\star}\right)^{\top} \left((A-{\mu\cdot 1_{nd}1_{nd}^\top}) \circ G^{\star}G^{\star\top} + (A \circ \Delta) \right) \left(G+G^{\star}\right)\right)   \nonumber \\
& \leq 2\sqrt{nd} \left(   \| (A-{\mu\cdot 1_{nd}1_{nd}^\top}) \circ  G^{\star}G^{\star \top} \| + \left\| A \circ \Delta \right\|  \right) {\textrm{d}}(G,G^\star),
\end{align}
where the last inequality is from $\Vert G+G^*\Vert_F\leq \Vert G\Vert_F+\Vert G^*\Vert_F\leq  2\sqrt{nd}$ and ${\rm d}(G,G^\star) =  \Vert G-G^\star\Vert_F$. Now we estimate the term $\| (A-{\mu\cdot 1_{nd}1_{nd}^\top}) \circ  G^{\star}G^{\star \top} \|$. Let $\bm{u}:=[\bm{u}_1;\ldots;\bm{u}_n]\in\mathbb{R}^{nd}$ and $\bm{v}:=[\bm{v}_1;\ldots;\bm{v}_n]\in\mathbb{R}^{nd}$ where $\bm{u}_i,\bm{v}_i\in\mathbb{R}^d$ for each $i\in[n]$. From the property of the operator norm \cite[Theorem 5.6.2(d)]{horn2012matrix} we know that
$$\left\| (A-{\mu\cdot 1_{nd}1_{nd}^\top}) \circ  G^{\star}G^{\star \top} \right\| =\sup_{ \| \bm{u}\|_2= \| \bm{v} \|_2=1 } \left| \bm{u}^\top (A-{\mu\cdot 1_{nd}1_{nd}^\top})\circ  G^{\star}G^{\star \top} \bm{v}         \right|,$$
and this together with $A = W \otimes (1_d 1_d^\top)$ implies that
	\begin{align}
	  \left\| (A-{\mu\cdot 1_{nd}1_{nd}^\top}) \circ  G^{\star}G^{\star \top} \right\|  
	& =\sup_{ \| \bm{u}\|_2= \| \bm{v} \|_2=1 } \left| \sum_{1\leq i \neq j \leq n}  (W_{ij}-\mu) \bm{u}_i^\top G_i^\star G_j^{\star\top} \bm{v}_j          \right| \notag \\
	& = \sup_{\| \bm{u}\|_2= \| \bm{v} \|_2=1} \left| \sum_{1\leq i \neq j \leq n}  (W_{ij}-\mu) \bm{u}_i^\top  \bm{v}_j          \right| \notag \\
	& =  \sup_{\| \bm{u}\|_2= \| \bm{v} \|_2=1} \left|  \sum_{1\leq i \neq j \leq n} \bm{u}_i^\top \left( (W_{ij}-\mu)I_d  \right)    \bm{v}_j          \right| \notag \\
	& = \| (W - {\mu \cdot 1_n1_n^\top}) \otimes I_d \| = \|W - {\mu \cdot 1_n1_n^\top}\|   \label{eq: agstargstarnorm}.
\end{align}
Thus, combining the fact
\begin{equation*}
  \textrm{tr}(G^{\top} G^{\star}G^{\star\top} G)\leq \Vert G^{\top} G^{\star}\Vert \Vert G^{\top} G^{\star}\Vert_*  \leq n \Vert G^{\top} G^{\star}\Vert_*
\end{equation*}
with \eqref{eq: agstargstarnorm} and \eqref{dist-GstarGhat-key1}, it concludes that
\begin{align*}
 \frac{n\mu}{2}{\rm d}^{2}(G, G^{\star})
 &=\frac{n\mu}{2} \Vert G-G^\star\Vert_F^2=n\mu \left(n d-\| G^{\top} G^{\star}\|_{*} \right)\notag\\
&\leq 2\sqrt{nd} \left(   \| W-{\mu \cdot 1_n1_n^\top} \| + \left\| A \circ \Delta \right\|  \right) {\rm d}(G,G^\star),
\end{align*}
which implies that \eqref{dist-GstarGhat-ineq} holds. The proof is complete.
\end{proof}

Since the global maximizer $\hat{G}$ is close to the ground truth $G^\star$ from the conclusion of Lemma \ref{dist-GstarGhat},
 a natural further conclusion is that $\| G^{\star\top} \hat{G} \|_*$ as well as the singular values of $G^{\star\top} \hat{G}$ can be nontrivially lower bounded.
\begin{lemma}
  Let $\hat{G}\in\odn$ be a global maximizer of problem \eqref{Problem}. Then the singular values of $G^{\star\top} \hat{G}$ satisfy
  \begin{align} \label{eq: Gstarhatsingularvalue}
  	n- \frac{8d \left(   \left\| W-{\mu \cdot 1_n1_n^\top} \right\| + \left\| A \circ \Delta \right\|  \right)^2}{n\mu^2} \leq \sigma_{l} \left( G^{\star\top} \hat{G} \right) \leq n, 
  \end{align}
where $l \in [d]$. Furthermore,  for each $i \in [n]$, the smallest singular value of $C_{i}^\top \hat{G}$ satisfies
\begin{align}\label{singular-lowbd}
\sigma_{\min }(C_{i}^\top \hat{G})\ge n\mu- \frac{8d \left(   \left\| W-{\mu \cdot 1_n1_n^\top} \right\| + \left\| A \circ \Delta \right\|  \right)^2}{n\mu} -\| (A \circ \Delta)\hat{G} \|_\infty -\left\| \left((A - {\mu\cdot 1_{nd}1_{nd}^\top} ) \circ G^\star G^{\star\top} \right) \hat{G} \right\|_{\infty}.
\end{align}
\end{lemma}

\begin{proof}
From \eqref{dist-GstarGhat-ineq} we know that
\begin{align}\label{dist-key-4.1}
{\rm d}(G^\star,\hat{G})^2
& =\min\limits_{g\in\od} \Vert G^{\star}g-\hat{G}\Vert_F^2
 =\Vert G^\star\Vert_F^2+\Vert \hat{G}\Vert_F^2-2 \max\limits_{g\in\od} \langle G^{\star}g, \hat{G}\rangle\notag\\
&=2 \left(nd-\Vert G^{\star\top} \hat{G}\Vert_* \right)\leq   \frac{16d \left(   \left\| W-{\mu \cdot 1_n1_n^\top} \right\| + \left\| A \circ \Delta \right\|  \right)^2}{n\mu^2}.
\end{align}
Since $\|G^{\star\top} \hat{G}\| \leq \sum_{i=1}^n \|G_i^{\star\top} \hat{G}_i\| \leq n$ implies $0 \leq \sigma_{l}(G^{\star\top} \hat{G}) \leq n$ for $l \in [d]$, we know from \eqref{dist-key-4.1} that
\begin{equation*}
n - \sigma_{\min }(G^{\star\top} \hat{G}) \leq \sum_{l=1}^{d}\left(n-\sigma_{l}(G^{\star\top} \hat{G})\right)=nd-\Vert G^{\star\top} \hat{G}\Vert_* \leq   \frac{8d \left(   \left\| W-{\mu \cdot 1_n1_n^\top} \right\| + \left\| A \circ \Delta \right\|  \right)^2}{n\mu^2},
\end{equation*}
which illustrates that $\sigma_{\min }(G^{\star\top} \hat{G}) \geq n- \frac{8d \left(   \left\| W-{\mu \cdot 1_n1_n^\top} \right\| + \left\| A \circ \Delta \right\|  \right)^2}{n\mu^2}$.
This together with the definition $C_i^\top \hat{G}=\mu G_i^\star G^{\star\top} \hat{G} + (A \circ \Delta)_i^\top \hat{G} + \left((A - {\mu\cdot 1_{nd}1_{nd}^\top} ) \circ G^\star G^{\star\top} \right)_i^\top \hat{G}$ leads to
\begin{align*}
& \sigma_{\min }(C_{i}^\top \hat{G})  \geq \sigma_{\min }(\mu G^{\star\top}\hat{G})- \left\|(A \circ \Delta)_i^\top \hat{G} + \left((A - {\mu\cdot 1_{nd}1_{nd}^\top} ) \circ G^\star G^{\star\top} \right)_i^\top \hat{G} \right \|  \\
& \ge n{\mu}- \frac{8d \left(   \left\| W-{\mu \cdot 1_n1_n^\top} \right\| + \left\| A \circ \Delta \right\|  \right)^2}{n\mu} -\| (A \circ \Delta)\hat{G} \|_\infty -\left\| \left((A - {\mu\cdot 1_{nd}1_{nd}^\top} ) \circ G^\star G^{\star\top} \right) \hat{G} \right\|_{\infty},
 \end{align*}
where the first inequality is from the Weyl's inequality. This completes the proof.
\end{proof}


Now we are ready to present the following main theorem of this section. To proceed, we define the residual function for proving the local error bound property. Recall that $\alpha\ge0$ and $\tilde{C}=C+\alpha I_{nd}$. Let $D_\alpha:\odn \rightarrow \mathbb{S}^{nd}$ be defined as
  \begin{equation} \label{eq: defineDG}
  		  D_\alpha(G):={\rm Diag} \left( \left[U_{\tilde{C}_1^\top G}\Sigma_{\tilde{C}_1^\top G}U_{\tilde{C}_1^\top G}^\top;\ldots;U_{\tilde{C}_n^\top G}\Sigma_{\tilde{C}_n^\top G}U_{\tilde{C}_n^\top G}^\top \right] \right)-\tilde{C},
  \end{equation} 
where $U_{{\tilde{C}_i^\top} G}\in \Xi({\tilde{C}_i^\top}G)_1$,
and then we define $\rho_\alpha:\odn \rightarrow \mathbb{R}_+$ as
  \begin{equation}\label{def-rho}
\rho_\alpha(G):=\Vert D_\alpha(G)G\Vert_F.
  \end{equation}
The operator $D_\alpha$ defined in \eqref{eq: defineDG} is a single-valued rather than a multi-valued mapping, since from \cite[Theorem 7.2.6]{horn2012matrix} we know that for any matrix $X\in\mathbb{R}^{m\times n}$, there exists a unique positive semidefinite matrix $(XX^\top)^{1/2}=U_{X}\Sigma_X U_X^\top$ for any $U_{X}\in \Xi(X)_1$.\\

Now we are ready to state the local error bound theorem as the main result for this section.
\begin{theorem}[Local error bound]\label{EB}
Suppose that \\
{\rm (i)} $\|W - {\mu \cdot 1_n1_n^\top} \| + \|  A \circ \Delta\|  \leq \frac{n^{3/4}\mu}{40 d^{1/2}}$ and $\| (A \circ \Delta)G^\star  \|_\infty \leq \frac{n{\mu}}{10}$, \\
{\rm (ii)} $\max_{i\in [n]} \left\| \left((A - {\mu\cdot 1_{nd}1_{nd}^\top} ) \circ G^\star G^{\star\top} \right)_i^\top \hat{G} \right\|  \leq  \frac{n{\mu}}{10}$,\\
{\rm (iii)}  the stepsize $\alpha \leq \frac{n{\mu}}{20\sqrt{2}}$.\\
Then for any $G\in\odn$ satisfying ${\rm d}(G,G^\star) \leq \frac{\sqrt{n}}{5}$ and any global maximizer $\hat{G}\in\odn$ of problem \eqref{Problem}, we have
\begin{equation}\label{errorbound-ineq}
  \frac{n{\mu}}{10}{\rm d}(G,\hat{G})\leq\rho_\alpha(G),
\end{equation}
where the residual function $\rho_\alpha$ is defined by \eqref{def-rho}.
\end{theorem}
Before proving Theorem \ref{EB}, we first give some explanation. 
The aforementioned computable upper bound $\rho_\alpha(\cdot)$ is actually constructed through the fixed points of $\mathcal{T}_\alpha$ from the characterization in Lemma \ref{lemma-equi}{\rm (b)} (see Remark \ref{remark-D} for more details).
It can be concluded that under these three conditions, up to a generalized global phase, there is a unique fixed point in the neighborhood of $G^\star$ (i.e., ${\rm d}(G,G^\star) \leq \frac{\sqrt{n}}{5}$), which is exactly the global maximizer. Thus, we can view $\rho_\alpha$ as a surrogate measure of global optimality.
\begin{proof}
The roadmap of our proof is to derive a proper lower bound for $ \Vert D_\alpha(\hat{G})G\Vert_F$ as well as an upper bound for $\Vert (D_\alpha (G)-D_\alpha(\hat{G}))G\Vert_F$, then from the following triangle inequality
\begin{equation}\label{EB-key000}
	\rho_\alpha(G) = \| D_\alpha(G)G \|_F  \ge \Vert D_\alpha(\hat{G})G\Vert_F- \Vert (D_\alpha(G)-D_\alpha(\hat{G}) ) G \Vert_F,
\end{equation}
the desired result will be obtained.
To that end, let $\hat{g}, \hat{g}^\star\in \od$ satisfy 
$${\rm d}(G,\hat{G})=\Vert G\hat{g}-\hat{G}\Vert_F,\quad  {\rm d}(\hat{G},G^\star) = \| \hat{G}\hat{g}^\star -G^\star\|_F,$$ 
which implies
\begin{equation*}
\Vert G\hat{g}-\hat{G}\Vert_F^2=2(nd-\Vert \hat{G}^\top G\Vert_*), \quad \Vert \hat{G}\hat{g}^\star -G^\star\Vert_F^2=2(nd-\Vert \hat{G}^\top G^\star\Vert_*).
\end{equation*} 
Let $\hat{R}:=\Pi_{{\rm col}(\hat{G})^\perp}(G\hat{g} - \hat{G})=(I-\frac{1}{n}\hat{G}\hat{G}^\top)(G\hat{g} - \hat{G})$ be the projection of $G\hat{g} - \hat{G}$ onto the orthogonal complement of the column space of $\hat{G}$. 

From Lemma \ref{prop: globalfix} we know that $\hat{G}$ is a fixed point of $\mathcal{T}_\alpha$ and $C_i^\top \hat{G} \hat{G}_i\succeq \mu I_d$ is symmetric and positive definite for each $i\in[n]$.  Then it follows from Lemma \ref{lemma-equi}(b) and $\tilde{C} = C + \alpha I_{nd}$ that
\begin{equation*}
U_{\tilde{C}_i^\top \hat{G}}\Sigma_{\tilde{C}_i^\top \hat{G}}U_{\tilde{C}_i^\top \hat{G}}^\top=\tilde{C}_i^\top \hat{G}\hat{G}_i^\top= C_i^\top\hat{G}\hat{G}_i^\top +\alpha I_d= U_{C_i^\top \hat{G}}\Sigma_{C_i^\top \hat{G}}U_{C_i^\top \hat{G}}^\top +\alpha I_d,
\end{equation*}
and consequently
\begin{align}\label{EB-key1}
\left\langle D_\alpha(\hat{G})\hat{R},\hat{R}\right\rangle
  & =\sum_{i=1}^n{\rm tr} \left(\hat{R}_i^\top U_{\tilde{C}_i^\top \hat{G}}\Sigma_{\tilde{C}_i^\top \hat{G}}U_{\tilde{C}_i^\top \hat{G}}^\top\hat{R}_i \right)-{\rm tr}(\hat{R}^\top \tilde{C} \hat{R})\notag\\
  &=\sum_{i=1}^n{\rm tr} \left(\hat{R}_i^\top U_{C_i^\top \hat{G}}\Sigma_{C_i^\top \hat{G}}U_{C_i^\top \hat{G}}^\top\hat{R}_i \right)-{\rm tr}(\hat{R}^\top C\hat{R})\notag\\
  &\ge\sum_{i=1}^n \sigma_{\min}({C_i^\top \hat{G}})\cdot \Vert U_{C_i^\top\hat{G}}^\top\hat{R}_i\Vert_F^2-{\rm tr} \left(\hat{R}^\top (A \circ (G^\star G^{\star\top}+\Delta) )\hat{R} \right)\notag\\
  &\ge \min\limits_{i \in [n]}\left\{\sigma_{\min}({C_i^\top \hat{G}}) \right\} \cdot \Vert\hat{R}\Vert_F^2 - {\mu}\cdot\Vert \hat{R}^{\top} G^{\star} \Vert_F^2-\left( \| A \circ \Delta\Vert  + \| W - {\mu \cdot 1_n1_n^\top}  \| \right) \Vert \hat{R}\Vert_F^2.
  \end{align}
Since $\hat{G}^\top\hat{R}=\bm{0}$ by the definition of $\hat{R}$, we have
\begin{align} \label{eq: hatRGstar}
	\Vert \hat{R}^{\top} G^{\star} \Vert_F^2=\Vert \hat{R}^\top(G^\star-\hat{G}\hat{g}^\star)\Vert_F^2\leq {\rm d}^2(G^\star,\hat{G})\cdot \Vert \hat{R}\Vert_F^2.
\end{align}
Upon substituting \eqref{eq: hatRGstar} into \eqref{EB-key1}, we obtain
  \begin{align}\label{EB-key2}
   & \left\langle D_\alpha(\hat{G})\hat{R},\hat{R}\right\rangle \notag \\
   & \geq \min\limits_{i \in [n]}\left\{\sigma_{\min}({C_i^\top \hat{G}}) \right\} \cdot \Vert \hat{R}\Vert_F^2 - {\mu}{\rm d}^2(G^\star,\hat{G})\cdot\Vert \hat{R}\Vert_F^2-\left( \| A \circ \Delta\Vert  + \| W - {\mu \cdot 1_n1_n^\top} \| \right) \Vert \hat{R}\Vert_F^2 \notag \\
  &\ge\left(n{\mu}- \frac{8d \left(   \left\| W-{\mu \cdot 1_n1_n^\top} \right\| + \left\| A \circ \Delta \right\|  \right)^2}{n{\mu}} -\| (A \circ \Delta) \hat{G} \|_\infty -   \left\| \left((A - {\mu\cdot 1_{nd}1_{nd}^\top} ) \circ G^\star G^{\star\top} \right) \hat{G} \right\|_{\infty} \right) \Vert \hat{R}\Vert_F^2 \notag \\
  & \quad - \left( \| A \circ \Delta\Vert  + \left\| W - {\mu \cdot 1_n1_n^\top} \right\| + {\mu}{\rm d}^2(G^\star,\hat{G}) \right)\Vert \hat{R}\Vert_F^2,
\end{align}
where the second inequality is from \eqref{singular-lowbd}. Since
\begin{align}\label{EB-ineq-key}
  \Vert \hat{G}^\top (G\hat{g} - \hat{G})\Vert_F
  &= \Vert \hat{G}^\top G\hat{g} - nI_d\Vert_F= \sqrt{\Vert\hat{G}^\top G\Vert_F^2-2n\Vert\hat{G}^\top G\Vert_*+n^2 d}\notag\\
  &= \sqrt{\sum_{i=1}^d \left(n-\sigma_i(\hat{G}^\top G) \right)^2}\leq \sqrt{ \left(nd-\Vert\hat{G}^\top G\Vert_* \right)^2}=\frac{1}{2}{\rm d}^2(G,\hat{G}),
  \end{align} 
we know from the definition of $\hat{R}$ that
  \begin{align}
    \Vert \hat{R}\Vert_F 
    &\ge \Vert G\hat{g} - \hat{G}\Vert_F -\frac{1}{n}\Vert \hat{G}\hat{G}^\top(G\hat{g} - \hat{G})\Vert_F\notag\\
    &\ge{\rm d}(G,\hat{G})-\frac{1}{\sqrt{n}}\Vert \hat{G}^\top(G\hat{g} - \hat{G}) \Vert_F
    ={\rm d}(G,\hat{G})-\frac{1}{2\sqrt{n}}{\rm d}^2(G,\hat{G}).\label{eq: hatRdist}
    \end{align} 
Moreover, since $\hat{G}$ is a fixed point of $\mathcal{T}_\alpha$ from Lemma \ref{prop: globalfix}, we know that $D_\alpha(\hat{G})\hat{G}=\bm{0}$ from Lemma \ref{lemma-equi}(b).
    Combining this with \eqref{EB-key2}, \eqref{eq: hatRdist} and \eqref{dist-GstarGhat-ineq} leads to
  \begin{align}\label{EB-1}
  &\Vert D_\alpha(\hat{G})G\Vert_F=\Vert D_\alpha(\hat{G})\hat{R}\Vert_F\ge \frac{\langle D_\alpha(\hat{G})\hat{R},\hat{R}\rangle}{\Vert \hat{R}\Vert_F}\notag\\
    &\ge\left(n{\mu} -\| (A \circ \Delta) \hat{G} \|_\infty -   \left\| \left((A - {\mu\cdot 1_{nd}1_{nd}^\top} ) \circ G^\star G^{\star\top} \right) \hat{G} \right\|_{\infty} \right)  \left( {\rm d}(G,\hat{G})-\frac{1}{2\sqrt{n}}{\rm d}^2(G,\hat{G}) \right) \notag \\
    & - \left( \| A \circ \Delta\Vert_2  + \left\| W - {\mu \cdot 1_n1_n^\top} \right\| +\frac{24d \left(   \left\| W-{\mu \cdot 1_n1_n^\top} \right\| + \left\| A \circ \Delta \right\|  \right)^2}{n{\mu}}  \right)  \left( {\rm d}(G,\hat{G})-\frac{1}{2\sqrt{n}}{\rm d}^2(G,\hat{G}) \right).
    \end{align}
  Next, since $\Xi({\tilde{C}_i^\top}G)_1=\Xi({\tilde{C}_i^\top}G\hat{g})_1$ and $\Sigma_{{\tilde{C}_i^\top}G}=\Sigma_{{\tilde{C}_i^\top}G\hat{g}}$,
   we know that for any $U_{{\tilde{C}_i^\top} G\hat{g}}\in \Xi({\tilde{C}_i^\top}G\hat{g})_1$ it follows that
  \begin{align}\label{EB-key3}
    \Vert  (D_\alpha(G)-D_\alpha(\hat{G}))G\Vert_F
    &=\Vert(D_\alpha(G\hat{g})-D_\alpha(\hat{G}))G\Vert_F\notag\\
    &=\sqrt{\sum_{i=1}^n \Big\Vert U_{{\tilde{C}_i^\top} G\hat{g}}\Sigma_{{\tilde{C}_i^\top} G\hat{g}}U_{{\tilde{C}_i^\top} G\hat{g}}^\top-U_{{\tilde{C}_i^\top} \hat{G}}\Sigma_{{\tilde{C}_i^\top} \hat{G}}U_{{\tilde{C}_i^\top} \hat{G}}^\top\Big\Vert_F^2} \notag \\
    & =\sqrt{\sum_{i=1}^n \Big\Vert ({\tilde{C}_i^\top} G\hat{g}\hat{g}^\top G^\top {\tilde{C}_i})^{\frac{1}{2}}-({\tilde{C}_i^\top} \hat{G} \hat{G}^\top {\tilde{C}_i})^{\frac{1}{2}} \Big\Vert_F^2} \notag\\
    &\leq \sqrt{\sum_{i=1}^n 2\Vert {\tilde{C}_i^\top} (G\hat{g} - \hat{G})\Vert_F^2}= \sqrt{2} \Vert \tilde{C} (G\hat{g} - \hat{G})\Vert_F,
    \end{align}
where the inequality is due to Fact \ref{lemma: lipmatmodu}.   
Now, we estimate $\Vert \tilde{C} (G\hat{g} - \hat{G})\Vert_F$. 
Since we know from \eqref{EB-ineq-key} and \eqref{dist-GstarGhat-ineq} that
\begin{align}
  \Vert G^{\star\top}(G\hat{g} - \hat{G})\Vert_F
  &\le \Vert (G^\star-\hat{G}\hat{g}^\star)^\top (G\hat{g} - \hat{G})\Vert_F +\Vert \hat{G}^\top (G\hat{g} - \hat{G})\Vert_F\notag\\
  &\le \Vert G^\star-\hat{G}\hat{g}^\star\Vert_F\Vert G\hat{g} - \hat{G}\Vert_F +\Vert \hat{G}^\top (G\hat{g} - \hat{G})\Vert_F\notag\\
  & \leq \frac{4\sqrt{d} \left(   \left\| W-{\mu \cdot 1_n1_n^\top} \right\| + \left\| A \circ \Delta \right\|  \right) }{\sqrt{n}{\mu}} {\rm d}(G,\hat{G})+\frac{1}{2}{\rm d}^2(G,\hat{G}),\notag
  \end{align} 
it follows that
\begin{align}\label{EB-key4}
  & \Vert \tilde{C} (G\hat{g} - \hat{G})\Vert_F   \notag \\
  & \leq \left\| \left(A \circ G^\star G^{\star\top}\right) (G\hat{g} - \hat{G}) \right\|_F + (\| A \circ \Delta \| + \alpha) {\rm d}(G,\hat{G})     \notag \\
  & \leq \left\| \left((A - {\mu\cdot 1_{nd}1_{nd}^\top}) \circ G^\star G^{\star\top}\right) (G\hat{g} - \hat{G}) \right\|_F + (\| A \circ \Delta \| + \alpha) {\rm d}(G,\hat{G})  +   {\mu}\sqrt{n}\Vert G^{\star\top}(G\hat{g}-\hat{G})\Vert_F  \notag \\
    &\leq 4\sqrt{d} \left(   \| W-{\mu \cdot 1_n1_n^\top} \| + \left\| A \circ \Delta \right\|  \right) {\rm d}(G,\hat{G})+\frac{{\mu}}{2}\sqrt{n}{\rm d}^2(G,\hat{G}) \notag\\
	&\quad +\left(\| A \circ \Delta \| + \alpha + \| W - {\mu \cdot 1_n1_n^\top}\|\right) {\rm d}(G,\hat{G})\notag\\
    &\leq \left(5\sqrt{d} \left(   \| W-{\mu \cdot 1_n1_n^\top} \| + \left\| A \circ \Delta \right\|  \right)+\alpha \right){\rm d}(G,\hat{G})+\frac{{\mu}}{2}\sqrt{n}{\rm d}^2(G,\hat{G}).
\end{align}
Then \eqref{EB-key3} and \eqref{EB-key4} imply that
\begin{align}\label{EB-2}
\Vert (D_\alpha(G)-D_\alpha(\hat{G}))G\Vert_F 
  &\leq  \sqrt{2}\left(5\sqrt{d} \left(   \| W-{\mu \cdot 1_n1_n^\top} \| + \left\| A \circ \Delta \right\|  \right)+\alpha \right){\rm d}(G,\hat{G})+\frac{{\mu}\sqrt{2n}}{2}{\rm d}^2(G,\hat{G}).
\end{align}

Now, by the assumption that $\textrm{d}(G,G^\star) \leq \frac{\sqrt{n}}{5}$ and \eqref{dist-GstarGhat-ineq}, we have
\begin{align*}
	\textrm{d}(G,\hat{G}) \leq \textrm{d}(G,G^\star) + \textrm{d}(\hat{G}, G^\star) \leq \frac{\sqrt{n}}{5} +  \frac{4\sqrt{d} \left(   \| W-{\mu \cdot 1_n1_n^\top} \| + \left\| A \circ \Delta \right\|  \right) }{\sqrt{n}{\mu}},
\end{align*}
which implies that 
\begin{align} \label{eq: dsquaregghat}
	\textrm{d}^2(G,\hat{G})  \leq \left( \frac{\sqrt{n}}{5} +  \frac{4\sqrt{d} \left(   \| W-{\mu \cdot 1_n1_n^\top} \| + \left\| A \circ \Delta \right\|  \right) }{\sqrt{n}{\mu}} \right)\textrm{d}(G,\hat{G}) .
\end{align}
In addition,
\begin{align} \label{eq: errorinfty}
	\| (A \circ \Delta)\hat{G} \|_\infty &\leq \| (A \circ \Delta)(\hat{G}\hat{g}^\star - G^\star)\|_\infty + 	\| (A \circ \Delta) G^\star \|_\infty \nonumber \\
	& \leq \| (A \circ \Delta) G^\star \|_\infty + \| A \circ \Delta \| \textrm{d}(\hat{G},G^\star) \nonumber \\
	& \leq \| (A \circ \Delta) G^\star \|_\infty + \frac{4\sqrt{d} \left(   \| W-{\mu \cdot 1_n1_n^\top} \| + \left\| A \circ \Delta \right\|  \right)^2 }{\sqrt{n}{\mu}}.
\end{align}
Thus, it follows from \eqref{EB-key000}, \eqref{EB-1}, \eqref{EB-2}, \eqref{eq: dsquaregghat} and \eqref{eq: errorinfty} that 
\begin{align*}
& \rho_\alpha(G) \\
& \ge
\left( n{\mu} - \|(A \circ \Delta)G^\star\|_\infty-\frac{4\sqrt{d}K^2}{\sqrt{n}{\mu}}-L-K-\frac{24dK^2}{n{\mu}}        \right) \left( 1 - \frac{1}{2\sqrt{n}} \left(\textrm{d}(G,G^\star) + \frac{4\sqrt{d}K}{\sqrt{n}{\mu}} \right) \right) \textrm{d}(G,\hat{G})\\
&\quad - \left(5\sqrt{2d}K + \sqrt{2}\alpha \right) \textrm{d}(G,\hat{G}) - \frac{{\mu}\sqrt{2n}}{2}\left( \textrm{d}(G,G^\star) + \frac{4\sqrt{d}K}{\sqrt{n}{\mu}} \right) \textrm{d}(G,\hat{G}) \\
& \geq  \frac{n\mu}{10} {\rm d}(G,\hat{G}),
\end{align*}
where $K :=\| W-{\mu \cdot 1_n1_n^\top} \| + \| A \circ \Delta \|$ and $L := \| ((A - {\mu\cdot 1_{nd}1_{nd}^\top} ) \circ G^\star G^{\star\top} ) \hat{G} \|_{\infty} $ for simplicity and the last inequality is from the assumptions (i)-(iii). The proof is complete.
\end{proof}

\begin{remark}\label{remark-D}
For any $G\in\odn$ and $T_{\alpha}(G)\in \mathcal{T}_\alpha(G)$ one has that 
$$D_\alpha(G)G={\rm Diag}(\tilde{C}G)\cdot {\rm Diag}\left( \left[ \left(T_\alpha(G)-G\right)_1^\top;\ldots;\left(T_\alpha(G)-G\right)_n^\top \right]\right)G,$$ 
and consequently we know that
$$\rho_\alpha(G)=\Vert D_\alpha(G)G\Vert_F\leq nd\Vert \tilde{C}\Vert \Vert G-T_\alpha(G)\Vert_F,$$ which illustrates that the residual function we use in Theorem \ref{EB} is controlled by the Frobenius norm of the term $G-T_\alpha(G)$. This motivates us to use the projected gradient ascent method to solve problem \eqref{Problem} from the theoretical perspective.

\end{remark}

Again, one may ask whether the local error bound property holds for problem \eqref{Problem-S}. The analysis of Theorem \ref{EB} highly relies on the fact that any global maximizer $\hat{G}$ satisfies $C_i^\top \hat{G} \hat{G}_i\succeq \bm{0}$ for each $i\in[n]$. Granting this, we can similarly prove the local error bound property for problem \eqref{Problem-S} with the same residual function $\rho_{\alpha}$. Actually, under a more restrictive noise level, such property can be promised by the following proposition for problem \eqref{Problem-S}.

\begin{prop}\label{sod-noise}
	Let $\hat{G}\in\sodn$ be a global maximizer of problem \eqref{Problem-S}. If $\| W-{\mu \cdot 1_n1_n^\top} \| + \left\| A \circ \Delta \right\|\leq \frac{2}{13}\mu n^{\frac{1}{2}}d^{-\frac{1}{2}}$, then $C_i^\top \hat{G}\hat{G}_i^\top\succeq\bm{0}$ for all $i\in[n]$.
  \end{prop}
  \begin{proof}
Since a global maximizer of problem \eqref{Problem-S} is a first-order critical point of problem \eqref{Problem}. Then from Lemma \ref{lemma-equi2}(a) we know that $C_i^\top \hat{G}\hat{G}_i^\top$ is symmetric, and consequently for all $i\in[n]$ it follows from the definition that 
\begin{align}\label{psd-key1}
  C_i^\top \hat{G}\hat{G}_i^\top
  &= \mu G_i^\star G^{\star\top} \hat{G}\hat{G}_i^\top + (A \circ \Delta)_i^\top \hat{G}\hat{G}_i^\top + \left((A - {\mu\cdot 1_{nd}1_{nd}^\top} ) \circ G^\star G^{\star\top} \right)_i^\top \hat{G}\hat{G}_i^\top.
  \end{align}
  Since $\Vert(A \circ \Delta)_i^\top \hat{G}\hat{G}_i^\top + \left((A - {\mu\cdot 1_{nd}1_{nd}^\top} ) \circ G^\star G^{\star\top} \right)_i^\top \hat{G}\hat{G}_i^\top \Vert\leq\sqrt{n}\left(\left\| A \circ \Delta \right\|+\left\| W-{\mu \cdot 1_n1_n^\top} \right\| \right)$, it follows from \eqref{psd-key1} and Weyl's inequality that
  \begin{align}\label{psd-main}
  \lambda_{\min}(C_i^\top \hat{G}\hat{G}_i^\top)
  &\ge \lambda_{\min}\left(\frac{1}{2}\mu(G^\star_i G^{\star\top}\hat{G}\hat{G}_i^\top+\hat{G}_i\hat{G}^\top G^\star G_i^{\star\top})\right)-\sqrt{n}\left(\left\| A \circ \Delta \right\|+\| W-{\mu \cdot 1_n1_n^\top} \| \right).
  \end{align}
  Let $\hat{g}^\star\in \mathcal{O}(d)$ satisfy ${\rm d}(G^\star,\hat{G})=\Vert \hat{G}\hat{g}^\star-G^\star\Vert_F$, then
  \begin{align}\label{psd-key2}
   \Vert G ^\star_i G^{\star\top}\hat{G}\hat{G}_i^\top -nI_d\Vert
	&\le \Vert G ^\star_i G^{\star\top}\hat{G}\hat{G}_i^\top -n G ^\star_i \hat{g}^{\star\top}\hat{G}_i^\top\Vert+n\Vert G ^\star_i \hat{g}^{\star\top} \hat{G}_i^\top- I_d\Vert\notag\\
	&\le \Vert G^{\star\top}\hat{G}\hat{g}^\star -n I_d\Vert_*+n\Vert \hat{G}_i \hat{g}^\star- G ^\star_i\Vert_F\notag\\
	&\le {\rm tr}(n I_d-G^{\star\top}\hat{G}\hat{g}^\star)+n\Vert \hat{G} \hat{g}^\star- G ^\star\Vert_F \notag\\
	&\le \frac{1}{2}{\rm d}^2(G^\star,\hat{G})+n{\rm d}(G^\star,\hat{G})\notag\\
	&\le \frac{8d \left(   \left\| W-{\mu \cdot 1_n1_n^\top} \right\| + \left\| A \circ \Delta \right\|  \right)^2 }{n\mu^2}+\frac{4\sqrt{nd} \left(   \left\| W-{\mu \cdot 1_n1_n^\top} \right\| + \left\| A \circ \Delta \right\|  \right) }{\mu},
  \end{align}
  where the last inequality is from \eqref{dist-GstarGhat-ineq}. Again from Weyl's inequality and \eqref{psd-key2} one has that
  \begin{align}\label{psd-key3}
	&\lambda_{\min}\left(\frac{1}{2}(G^\star_i G^{\star\top}\hat{G}\hat{G}_i^\top+\hat{G}_i\hat{G}^\top G^\star G_i^{\star\top})\right)\notag\\
	&\ge \lambda_{\min}(nI_d)-\bigg\Vert\frac{1}{2}(G^\star_i G^{\star\top}\hat{G}\hat{G}_i^\top+\hat{G}_i\hat{G}^\top G^\star G_i^{\star\top})-nI_d\bigg\Vert\notag\\
	&\ge n-\Vert G ^\star_i G^{\star\top}\hat{G}\hat{G}_i^\top -nI_d\Vert\notag\\
	&\ge n-\frac{8d \left(   \left\| W-{\mu \cdot 1_n1_n^\top} \right\| + \left\| A \circ \Delta \right\|  \right)^2 }{n\mu^2}-\frac{4\sqrt{nd} \left(   \left\| W-{\mu \cdot 1_n1_n^\top} \right\| + \left\| A \circ \Delta \right\|  \right) }{\mu}.
  \end{align}
  Thus, from \eqref{psd-main}, \eqref{psd-key3} and $\left\| W-{\mu \cdot 1_n1_n^\top} \right\| + \left\| A \circ \Delta \right\|\leq \frac{2}{13}\mu n^{\frac{1}{2}}d^{-\frac{1}{2}}$ it follows that
  $$\lambda_{\min}(C_i \hat{G}\hat{G}_i^\top)\ge  n\mu-\frac{8d \left(   \left\| W-{\mu \cdot 1_n1_n^\top} \right\| + \left\| A \circ \Delta \right\|  \right)^2 }{n\mu}-5\sqrt{nd} \left(   \| W-{\mu \cdot 1_n1_n^\top} \| + \left\| A \circ \Delta \right\|  \right)\ge0.$$
  Thus, $C_i \hat{G}\hat{G}_i^\top$ for all $i\in[n]$ are positive semidefinite matrices.
  \end{proof}

\begin{remark}
Note that when we specialize the setting to Gaussian noise $\Delta=\sigma Z$, where $Z\in\mathbb{S}^{nd}$ 
is a symmetric Gaussian random matrix, it follows that $\Vert Z\Vert\lesssim n^{\frac{1}{2}}d^{-\frac{1}{2}}$ with high probability as a standard result. Thus, for problem \eqref{Problem-S}, the above proposed routine to derive local error bound property can only allow constant noise level for $\sigma$, since it needs $\left\| \Delta \right\|\lesssim  n^{\frac{1}{2}}d^{-\frac{1}{2}}$ under the complete measurements case.




\end{remark}

\section{Linear convergence of the GPM}

In this section, we focus on the projected gradient ascent method (i.e., GPM) for problem \eqref{Problem}.
The GPM, originally introduced in \cite{journee2010generalized} for maximizing a convex function over a compact set, plays an important role in various general settings afterwards \cite{boumal2016nonconvex,liu2017estimation,zhong2018near,shen2020complete,zhai2020complete}. 

One prominent feature of the GPM for problem \eqref{Problem} \cite{ling2020improved,liu2020unified} is its two-stage framework. In the first stage (also called the spectral initialization), it first relaxes the original constraints $G \in \odn$ to $G^\top G = nI_d$ and then project the acquired solution onto $\odn$ to get a feasible point $G^0$ (the spectral estimator). Such a relaxation brings at least two advantages. The relaxed problem is easy to solve as the global maximizer of it equals to the top $d$ eigenvectors of $C$. Besides, under certain conditions on the noise and the measurement graph, it can be proved that $G^0$ is quite close to the ground truth $G^\star$. 

Based on such a high-quality initial point $G^0$, in the second stage, the GPM successively refines the obtained points via simple and low-complexity iterations. More precisely, at each iteration, it first performs a standard gradient step in Euclidean space and then a projection operation onto the product manifold $\odn$, which admits a closed-form solution. The whole framework of the GPM is summarized in Algorithm \ref{alg: gpm}. Note that when $\alpha = 0$, the GPM reduces to the classical power method, from which Algorithm \ref{alg: gpm} is named.
\begin{algorithm}[h] 
	\caption{GPM for problem \eqref{Problem}}
	\begin{algorithmic}[1]\label{alg: gpm}
		\STATE Input: The matrix $C$, stepsize $\alpha \geq 0$.\\
		\STATE Compute the top $d$ eigenvectors $\Phi$ of $C$ with $\Phi^\top \Phi = nI_d$.
		\STATE Compute $G^0 \in \Pi_\odn(\Phi)$. 
		\FOR{$k=0,1,\ldots$}
		\STATE $G^{k+1} \in \mathcal{T}_\alpha(G^k)= \Pi_{\odn}(\tilde{C}G^{k}),$ where $\tilde{C}:=C+\alpha I_{nd}$.
		\ENDFOR
	\end{algorithmic}
\end{algorithm}



In the following part of this section, we prove that both the sequence of iterates and the corresponding sequence of objective function values of Algorithm \ref{alg: gpm} converge linearly to a global maximizer and the optimal value of problem \eqref{Problem}. Such a result is significant since generally the GPM, a local optimization method, may not converge to a point we are interested, let alone the global maximizer with linear convergence rate. Although a similar result for problem \eqref{Problem} with complete measurements under Gaussian noise has appeared in \cite{ling2020improved}, we highlight that the technique in our proof is quite different from that in \cite{ling2020improved}.  They follow a similar routine in \cite{zhong2018near} and can fail when it comes to incomplete measurements and general noise models while our proof of the following linear convergence result of GPM is based on the local error bound result established in Theorem \ref{EB}. It is worth noting  that various significant error bound results have been widely used in the convergence analysis for different kind of algorithms and problems \cite{luo1993error,zhou2017unified,liu2019quadratic,zhang2020new}.
\begin{theorem}[Linear convergence rate of GPM] \label{thm:convergence-rate}
	Suppose that\\
	{\rm (i)} $\|W - {\mu \cdot 1_n1_n^\top} \| + \|  A \circ \Delta\|  \leq \frac{n^{3/4}{\mu}}{60 d^{1/2}}$ and $\| (A \circ \Delta)G^\star  \|_\infty \leq \frac{n{\mu}}{10}$, \\
	{\rm (ii)}  $\max_{i\in [n]} \left\| \left((A - {\mu\cdot 1_{nd}1_{nd}^\top} ) \circ G^\star G^{\star\top} \right)_i^\top \hat{G} \right\|  \leq  \frac{n{\mu}}{10}$,\\
	{\rm (iii)}  the stepsize $\alpha$ satisfies $ \| A \circ \Delta \| + \| W - {\mu \cdot 1_n1_n^\top}\| <\alpha \leq \frac{n{\mu}}{30\sqrt{2d}}$.\\ Then, the sequence $\{ G^k\}_{k\geq 0}$ generated by Algorithm \ref{alg: gpm} satisfies 
	\begin{align*}
		&f(\hat{G}) - f(G^k) \leq \left(f(\hat{G}) - f(G^0) \right) \lambda^k\quad \text{and}\quad {\rm d}(G^k,\hat{G}) \leq a \left(f(\hat{G}) - f(G^0)\right)^{1/2} \lambda^{k/2},	
	\end{align*}
	where 
$a > 0$, $\lambda \in (0,1)$ are constants that depend only on $n$, $d$, $\mu$, $\alpha$, and $\hat{G}$ is a global maximizer of problem \eqref{Problem}.
\end{theorem}

Before presenting the proof of Theorem \ref{thm:convergence-rate}, we need some preparations.
The result of Lemma \ref{dist-GstarGhat} can be employed to acquire the following lemma that quantifies the distance between the initial point $G^0$ and the ground truth $G^\star$.
\begin{lemma}[Spectral estimator]\label{lemma: spectral}
	The spectral estimator $G^0 \in \odn$ satisfies 
	\begin{align} 
		{\rm d}(G^0, G^\star) \leq \frac{8\sqrt{d} \left(   \left\| W-{\mu \cdot 1_n1_n^\top} \right\| + \left\| A \circ \Delta \right\|  \right) }{\sqrt{n}{\mu}}.
	\end{align}
\end{lemma}
\begin{proof}
Recall that $\Phi$ is the matrix of top $d$ eigenvectors of $C$ with $\Phi^\top \Phi = nI_d$, which satisfies ${\rm tr}(\Phi^\top C \Phi) \geq {\rm tr} \left(G^{\star\top} C G^{\star} \right)$.	Without loss of generality, we assume ${\rm  d}(\Phi, G^\star) = \| \Phi - G^\star \|_F$. Then it follows from Lemma \ref{dist-GstarGhat} that
	\begin{align*}
		{\rm d}(G^0, G^\star) \leq \| G^0 - G^\star \|_F \leq 2 \| \Phi - G^\star\|_F  \leq \frac{8\sqrt{d} \left(   \left\| W-{\mu \cdot 1_n1_n^\top} \right\| + \left\| A \circ \Delta \right\|  \right) }{\sqrt{n}{\mu}},
	\end{align*}
where the second inequality is due to \cite[Lemma 2]{liu2020unified}. The proof is complete.
\end{proof}

To render the error bound result of Theorem \ref{EB} applicable to the whole sequence of iterates $\{G^k\}_{k\geq 0}$, hence a workhorse in the subsequent proof of Theorem \ref{thm:convergence-rate}, we still need the following proposition guaranteeing that the entire sequence of iterates $\{G^k\}_{k\geq 0}$ lies close to the underlying ground truth $G^\star$. 
\begin{prop}[The sequence of iterates stays in the ball] \label{prop: stayintheball}
	Suppose that \\
	{\rm (i)} $\|W - {\mu \cdot 1_n1_n^\top} \| + \|  A \circ \Delta\|  \leq \frac{n{\mu}}{60 d^{1/2}}$,\\
	{\rm (ii)}  the stepsize $\alpha \leq \frac{n{\mu}}{30\sqrt{2d}}$.\\
Then the sequence of iterates $\{G^{k}\}_{k \geq 0}$ generated by Algorithm \ref{alg: gpm} satisfies ${\rm d}(G^k,G^\star) \leq \frac{\sqrt{n}}{5}$.
\end{prop}
\begin{proof}
	The result will be shown by induction. When $k=0$, we check by Lemma \ref{lemma: spectral} and the assumption {\rm (i)} that ${\rm d}(G^0,G^\star) \leq \frac{\sqrt{n}}{5}$. Next, we assume that ${\rm d}(G^k,G^\star)=\Vert G^k-G^\star g^k\Vert_F \leq \frac{\sqrt{n}}{5}$, where $g^k\in \mathcal{O}(d)$, for some $k \geq 0$. Then, by the iterative procedure of Algorithm \ref{alg: gpm} and \cite[Lemma 2]{liu2020unified}, we have
	\begin{align}\label{lemma-key1}
		& {\rm d}(G^{k+1},G^\star) =\min\limits_{g\in\mathcal{O}(d)}\Vert G^{k+1}-G^\star g\Vert_F\notag\\
		&\leq \Vert G^{k+1}-G^\star g^{k}\Vert_F=\Vert \Pi_{\odn}(\tilde{C}G^k)-G^\star g^{k}\Vert_F \leq 2\bigg\Vert \frac{1}{n{\mu}}\tilde{C}G^k-G^\star g^{k}\bigg\Vert_F \notag \\
		& =2\left\Vert \frac{1}{n} G^\star G^{\star \top} G^k-G^\star g^{k}\right\Vert_F+\frac{2}{n{\mu}} \left\Vert  \left((A - {\mu\cdot 1_{nd}1_{nd}^\top})\circ  G^\star G^{\star \top} + A\circ \Delta + \alpha I  \right) G^k \right\Vert_F\notag\\
		&\leq 2\left\Vert \frac{1}{n} G^\star G^{\star \top} G^k-G^\star g^{k}\right\Vert_F+\frac{2}{n{\mu}} \left\Vert  ((A - {\mu\cdot 1_{nd}1_{nd}^\top})\circ  G^\star G^{\star \top} + A\circ \Delta + \alpha I \right\Vert \cdot { \rm d}(G^k,G^\star) \notag \\
		&+ \frac{2}{n{\mu}} \left\Vert  \left((A - {\mu\cdot 1_{nd}1_{nd}^\top})\circ  G^\star G^{\star \top} + A\circ \Delta + \alpha I  \right) G^\star \right\Vert_F.
	\end{align}
	Note that
	\begin{align}\label{lemma-key12}
		\left\Vert \frac{1}{n} G^\star G^{\star \top} G^k-G^\star g^{k}\right\Vert_F^2 
		&=\sum_{i=1}^n \bigg\Vert \frac{1}{n}\sum_{j=1}^n G_i^\star G_j^{\star\top} G_j^k-G_i^\star g^{k} \bigg\Vert_F^2= n \bigg\Vert\frac{1}{n} \sum_{j=1}^n G_j^{\star\top} G_j^k-g^{k} \bigg\Vert_F^2\notag\\
		& =\frac{1}{n} \|G^{\star \top} G^{k}-n  g^{k}\|_{F}^{2}
		\leq\frac{1}{4n} {\rm d}^4(G^k,G^\star), 
	\end{align}
where the inequality is from \eqref{EB-ineq-key} with $G^\star$ and $G^k g^{k\top}$ replacing $\hat{G}$ and $G\hat{g}$, respectively. Then we know from \eqref{lemma-key1} and \eqref{lemma-key12} with \eqref{eq: agstargstarnorm} that (recall that $K=\| W-{\mu \cdot 1_n1_n^\top} \| + \| A \circ \Delta \|$)
	\begin{align}\label{lemma-key2}
		{\rm d}(G^{k+1},G^\star)
		&\leq \left(\frac{1}{\sqrt{n}}{\rm d}(G^k,G^\star)+\frac{2}{n{\mu}}(K+\alpha) \right){\rm d}(G^k,G^\star)+\frac{2}{n{\mu}}(K+\alpha)   \sqrt{nd} \notag \\
		& \leq \frac{\sqrt{n}}{5}.
	\end{align}
This completes the proof.
\end{proof}

Now, let us move to the second part of results, which depicts key properties of Algorithm \ref{alg: gpm}.
\begin{prop} \label{pro: 3conditioner}
Under the assumptions of Theorem \ref{thm:convergence-rate}, the sequence $\{G^{k}\}_{k \geq 0}$ generated by Algorithm \ref{alg: gpm} satisfies:\\
{\rm (a) (Sufficient ascent)} $f(G^{k+1})-f(G^{k}) \geq a_{0} \cdot\|G^{k+1}-G^{k}\|_{F}^{2}$,\\
{\rm (b) (Cost-to-go estimate)} $f(\hat{G})-f(G^{k}) \leq a_{1} \cdot {\rm d}^{2}(G^{k}, \hat{G})$,\\
{\rm (c) (Safeguard)} $\rho(G^{k}) \leq a_{2} \cdot\|G^{k+1}-G^{k}\|_{F}$,\\
where $a_{0}, a_{1}, a_{2}>0$ are constants that depend only on $n$, $d$, $\mu$, $\alpha$, and $\hat{G}$ is a global maximizer of problem \eqref{Problem}.
\end{prop}
\begin{proof}
(a): Recall that $\tilde{C} = C + \alpha I$, then we know
\begin{align}
  &f(G^{k+1})-f(G^{k}) = {\rm tr}\left((G^{k+1})^\top \tilde{C}G^{k+1}\right)  - {\rm tr}\left((G^{k})^\top \tilde{C}G^{k}\right) \notag \\
  &={\rm tr}\left((G^{k+1}-G^{k})^{\top} \tilde{C}(G^{k+1}-G^{k})\right)-2{\rm tr}\left((G^{k})^{\top} \tilde{C} G^{k}\right)+2 {\rm tr}\left((G^{k+1})^{\top} \tilde{C}G^{k}\right).\notag
\end{align}
We claim that ${\rm tr}((G^{k+1})^{\top} \tilde{C} G^{k}) \geq{\rm tr}((G^{k})^{\top} \tilde{C}G^{k})$ due to the observation that
  \begin{equation*}
  	  {\rm tr}\left((G^{k+1})^{\top} \tilde{C}G^{k}\right) =\sum_{i=1}^n {\rm tr}\left( V_{{\tilde{C}_i^\top} G^k}U_{{\tilde{C}_i^\top} G^k}^{\top} {\tilde{C}_i^\top} G^{k}\right) =\sum_{i=1}^n \Vert {\tilde{C}_i^\top} G^k\Vert_* \ge \sum_{i=1}^n{\rm tr}({\tilde{C}_i^\top} G^k G^{k\top}_i).
  \end{equation*}
  Hence, we conclude that
  $$
  f(G^{k+1})-f(G^{k}) \geq{\rm tr}\left((G^{k+1}-G^{k})^{\top} \tilde{C}(G^{k+1}-G^{k})\right) \geq a_{0} \cdot\|G^{k+1}-G^{k}\|_{F}^{2}
  $$
  with $a_{0}=\lambda_{\min} \left( A \circ \Delta + (A-{\mu\cdot 1_{nd}1_{nd}^\top})\circ G^\star G^{\star\top}+\alpha I_{nd} \right)>0$ which is from the lower bound of the stepsize $\alpha$ in assumption (iii) of Theorem \ref{thm:convergence-rate}. \\
  (b): Assume that for all $k \geq 0$, $\hat{g}^k\in \mathcal{O}(d)$ satisfies ${\rm d}(G^k,\hat{G})=\Vert G^k \hat{g}^k-\hat{G} \Vert_F$. Recall that $D:\odn \rightarrow \mathbb{S}^{nd}$ defined in \eqref{eq: defineDG}. From Lemma \ref{prop: globalfix} we know that $\hat{G}$ is a fixed point which implies $D(\hat{G})\hat{G}=\bm{0}$ from Lemma \ref{lemma-equi}(b). This together with
Lemma \ref{lemma-equi}(d) shows that
  \begin{align}
  f(\hat{G})-f(G^{k})
  &={\rm tr}(\hat{G}^{\top} \tilde{C} \hat{G})-{\rm tr}(G^{k\top} \tilde{C} G^k)=\sum\limits_{i=1}^n \Vert {\tilde{C}_i^\top} \hat{G}\Vert_*-{\rm tr}(G^{k\top} \tilde{C} G^k)\notag\\
  &={\rm tr}\left(G^{k\top} [{\rm Diag}([U_{\tilde{C}_1^\top \hat{G}}\Sigma_{\tilde{C}_1^\top \hat{G}}U_{\tilde{C}_1^\top \hat{G}}^\top;\dots;U_{\tilde{C}_n^\top \hat{G}}\Sigma_{\tilde{C}_n^\top \hat{G}}U_{\tilde{C}_n^\top \hat{G}}^\top])-\tilde{C}] G^{k} \right)\notag\\
  &={\rm tr}\left((G^{k}\hat{g}^k-\hat{G})^{\top} D(\hat{G})(G^{k}\hat{g}^k-\hat{G})\right) \notag\\
  &\leq\left(\|\tilde{C}\|+\|\tilde{C} \hat{G}\|_{\infty}\right) {\rm d}^{2}(G^{k}, \hat{G})\leq a_1\cdot{\rm d}^{2}(G^{k}, \hat{G})\notag
  \end{align}
where $a_1:=3n{\mu}$, since from \eqref{dist-GstarGhat-ineq}, \eqref{eq: agstargstarnorm} and the assumptions (i)-(iii) in Theorem \ref{thm:convergence-rate}, one has that
\begin{align}
   & \| \tilde{C} \| +  \Vert \tilde{C} \hat{G}\Vert_{\infty}  \leq \| C \| + \| C\hat{G} \|_\infty + 2 \alpha \notag \\
   & \leq n{\mu} + \| (A-{\mu\cdot 1_{nd}1_{nd}^\top}) \circ  G^{\star}G^{\star \top} \| + \| A \circ \Delta \|  + \| (A\circ G^\star G^{\star \top})\hat{G} \|_\infty  + \| (A \circ \Delta)\hat{G} \|_\infty +  2\alpha   \notag \\
      & \leq n{\mu} + \| W - {\mu \cdot 1_n1_n^\top}\| + \| A \circ \Delta \| + n{\mu} + \| ((A - {\mu\cdot 1_{nd}1_{nd}^\top})\circ G^\star G^{\star \top})\hat{G} \|_\infty    + \| (A \circ \Delta) G^\star \|_\infty\notag \\
   &\quad  + \frac{4\sqrt{d} \left(   \left\| W-{\mu \cdot 1_n1_n^\top} \right\| + \left\| A \circ \Delta \right\|  \right)\| A \circ \Delta \| }{\sqrt{n}\mu}+  2\alpha    \notag \\
      &\leq 3n{\mu}.
  \end{align} 
(c): From the definition we know that $D_\alpha$ is single-valued mapping which is not affected by different choices of singular value decomposition. Then by the definition of $G_i^{k+1}=U_{{\tilde{C}_i^\top} G^k}V_{\tilde{C}_i^\top G^k}^\top$ for arbitrary $(U_{{\tilde{C}_i^\top}G^k},V_{\tilde{C}_i^\top G^k})\in \Xi({\tilde{C}_i^\top}G)$, $i\in[n]$, we have
\begin{align}
  D(G^k)G^{k}&=[{\rm Diag}([U_{\tilde{C}_1^\top G^k}\Sigma_{\tilde{C}_1^\top G^k}U_{\tilde{C}_1^\top G^k}^\top;\dots;U_{\tilde{C}_n^\top G^k}\Sigma_{\tilde{C}_n^\top G^k}U_{\tilde{C}_n^\top G^k}^\top])-\tilde{C}]G^{k}\notag\\
  &={\rm Diag}([U_{\tilde{C}_1^\top G^k}\Sigma_{\tilde{C}_1^\top G^k}U_{\tilde{C}_1^\top G^k}^\top;\dots;U_{\tilde{C}_n^\top G^k}\Sigma_{\tilde{C}_n^\top G^k}U_{\tilde{C}_n^\top G^k}^\top])(G^{k}-G^{k+1})\notag
\end{align}
Then it follows that
\begin{align}
  \rho_\alpha(G^{k})=\Vert D_{\alpha}(G^k)G^k\Vert_F
  &=\Vert{\rm Diag}([U_{\tilde{C}_1^\top G^k}\Sigma_{\tilde{C}_1^\top G^k}U_{\tilde{C}_1^\top G^k}^\top;\dots;U_{\tilde{C}_n^\top G^k}\Sigma_{\tilde{C}_n^\top G^k}U_{\tilde{C}_n^\top G^k}^\top])(G^{k+1}-G^{k})\Vert_F \notag\\
  &\leq\|\tilde{C} G^{k}\|_{\infty}\|G^{k+1}-G^{k}\|_{F}\leq a_{2} \cdot\|G^{k+1}-G^{k}\|_{F}.\notag
\end{align}
where $a_2:=2n^{5/4}\mu$ is due to Proposition \ref{prop: stayintheball} and the assumptions (i)-(iii) in Theorem \ref{thm:convergence-rate} which implies that
\begin{align*}
	\| \tilde{C}G^k \|_\infty  & \leq \| (A \circ G^\star G^{\star\top})G^k \|_\infty + \| (A \circ \Delta)G^k\|_\infty + \alpha \\
	& \leq n{\mu} +   \frac{n{\mu}}{10} + \|  (A \circ \Delta) G^\star \|_\infty + \| A \circ \Delta\|{\rm d}(G^k,G^\star) + \alpha \\
	& \leq 2n^{5/4}{\mu}.
\end{align*}
The proof is complete.
\end{proof}


With Theorem \ref{EB}, Proposition \ref{prop: stayintheball} and Proposition \ref{pro: 3conditioner}, now we are ready to give the proof of Theorem \ref{thm:convergence-rate}.

\begin{proof}[Proof of Theorem \ref{thm:convergence-rate}]
By Theorem~\ref{EB}, Proposition \ref{prop: stayintheball} and~\ref{pro: 3conditioner}, it follows that
\begin{equation}\label{linear-key0}
f(\hat{G}) - f(G^k)\leq a_1\cdot \textrm{d}^2(G^k,\hat{G})\leq \frac{100a_1}{n^2{\mu}^2} \rho_{\alpha}(G^k)^2\leq \frac{100a_1a_2^2}{n^2{\mu}^2}\|G^{k+1}-G^k\|_F^2\leq \frac{100a_1a_2^2}{a_0n^2{\mu}^2}\left( f(G^{k+1})  - f(G^k) \right). 
\end{equation}
Then we have that 
\begin{align}\label{linear-key1}
f(\hat{G}) - f(G^{k+1}) 
&= f(\hat{G}) - f(G^k)  - \left( f(G^{k+1}) - f(G^k) \right)\notag\\
&\leq \left( \frac{100a_1a_2^2}{a_0n^2{\mu}^2} - 1 \right) \left( f(G^{k+1}) - f(\hat{G}) + f(\hat{G}) - f(G^k) \right).
\end{align}
		Since $f(\hat{G})\ge f(G^k)$ for $k\in\mathbb{N}$, we may assume without loss of generality that $a'=\tfrac{100a_1a_2^2}{a_0n^2{\mu}^2}>1$. Thus, from \eqref{linear-key1} one has that
		$$ f(\hat{G})-f(G^{k+1}) \le \frac{a'-1}{a'}\left( f(\hat{G}) - f(G^k) \right), $$
		which yields with $\lambda=\tfrac{a'-1}{a'}\in(0,1)$ that
		\begin{equation}\label{linear-key2}
		f(\hat{G})-f(G^k) \le \left( f(\hat{G})-f(G^0) \right) \lambda^k.
		\end{equation}
		Furthermore, from \eqref{linear-key0} and \eqref{linear-key2} we know that
		\begin{eqnarray*}
			{\rm d}^2(G^k,\hat{G}) 
			\le \frac{100a_2^2}{a_0n^2{\mu}^2}\left( f(\hat{G})-f(G^k) \right) 
			\le \frac{100a_2^2}{a_0n^2{\mu}^2} \left( f(\hat{G})-f(G^0) \right) \lambda^k,
		\end{eqnarray*}
		which implies that
		$$ \textrm{d}(G^k,\hat{G}) \le a \left( f(\hat{G})-f(G^0) \right)^{1/2} \lambda^{k/2} $$
		with $a:=\sqrt{\tfrac{100a_2^2}{a_0n^2{\mu}^2}}$. This completes the proof. 
\end{proof}

Next, we specialize Theorem~\ref{thm:convergence-rate} to the Gaussian noise setting with the measurement graph assumed to be the Erd\"os-R\'enyi graph $\mathcal{G}([n],p)$ satisfying: 
\begin{itemize}
\item 
$W_{ij}$ are i.i.d. random variables following the Bernoulli distribution taking $1$ with probability $p$ (associated with $n$), otherwise being $0$, and $W_{ji}= W_{ij}$ for each $i<j$.
\item $W_{ii}= \mu = \frac{\sum_{i<j}W_{ij}}{n(n-1)/2}$ for each $i\in [n]$.
\end{itemize}

	\begin{cor}[Specialization to Gaussian noise setting] \label{cor: Gaussianlinear}
		Suppose that \\
{\rm (i)} 
the measurement graph is the Erd\"os-R\'enyi graph $\mathcal{G}([n],p)$ with $p \geq  \frac{\kappa_0 d}{\sqrt{n}}$;\\ 		
		{\rm (ii)} 
		$\Delta=\sigma Z$,
where $0 <\sigma \leq \frac{\kappa_1n^{1/4}p^{1/2}}{d} $, $Z\in\mathbb{S}^{nd}$ with  $Z_{ii}=\bm{0}$ for $i\in[n]$ and $Z_{ij}$ are i.i.d. standard Gaussian variables for $i\neq j$;\\
{\rm (iii)} 
the stepsize $\frac{\kappa_0 n^{3/4}p}{d^{1/2}} \leq \alpha \leq \frac{\kappa_1 np}{d^{1/2}}$, \\
where $\kappa_0>0$ is a large constant and $\kappa_1>0$ is a small constant.
Then for sufficiently large $n\in\mathbb{N}$, the sequence of iterates $\{G^k\}_{k\ge0}$ generated by Algorithm~\ref{alg: gpm} with high probability satisfies
		\begin{eqnarray*}
			f(\hat{G}) - f(G^k) \le \left( f(\hat{G})-f(G^0) \right) \lambda^k\quad \text{and}\quad
			{\rm d}(\hat{G},G^k) \leq a \left( f(\hat{G})-f(G^0) \right)^{1/2} \lambda^{k/2},
		\end{eqnarray*}
		where $a>0$, $\lambda\in(0,1)$ are constants that depend only on $n$, $d$, $p$, $\alpha$, and $\hat{G}$ is any global maximizer of problem~\eqref{Problem}.
	\end{cor}

Before presenting the proof of Corollary \ref{cor: Gaussianlinear}, we compare it with the result in \cite[Theorem 3.2]{ling2020improved}, which shows convergence of the GPM (with $\alpha=0$) under the Gaussian noise level $\sigma=O(\frac{\sqrt{n}}{\sqrt{d}(\sqrt{d}+\sqrt{\log n})})$ with complete measurement graph setting. When Corollary \ref{cor: Gaussianlinear} reduced to complete measurement case, the Gaussian noise level our result can tolerate is $\sigma=O(\frac{n^{1/4}}{d})$ since it is a direct consequence of above determistic result while \cite{ling2020improved} utilizes the structural information of Gaussian noise through the leave-one-out technique and derives a better control of \eqref{eq: errorinfty}. The conclusion of Corollary \ref{cor: Gaussianlinear} will hold with the same noise regime as in \cite{ling2020improved} when applied with similar technique.


\begin{lemma}[Concentration inequalities]
	Under the assumptions of Corollary \ref{cor: Gaussianlinear}, there exists a constant $c_0>0$ such that for sufficiently large $n\in\mathbb{N}$, it follows that
	\begin{align}
		& |\mu - p | \leq  c_0 \sqrt{p}\log n/n   \label{eq: bernmean},\\
		& \left\| (A-{\mu\cdot 1_{nd}1_{nd}^\top}) \circ G^{\star}G^{\star \top}  \right\| \leq 2c_0 \sqrt{2n \mu}, \label{Concen-key1}\\
		&  \max_{i\in [n]} \left\| \left((A - {\mu\cdot 1_{nd}1_{nd}^\top} ) \circ G^\star G^{\star\top} \right)_i^\top \hat{G} \right\|  \leq  c_0 \sqrt{2nd \mu \log n}, \label{Concen-key2}\\
		& \| A \circ Z \| \leq c_0 \sqrt{2nd \mu} \label{Concen-key3}
	\end{align}
hold with probability at least $1-9n^{-10}$.
	
\end{lemma}
\begin{proof} Let $c_0>0$ be a sufficiently large constant to meet all requirements of the following analysis. Since $\mu-p = \sum_{i<j}\frac{2}{n(n-1)}(W_{ij}-p)$ is the sum of independent mean-zero random variables, and it follows that $$\left\vert\frac{2}{n(n-1)}(W_{ij}-p)\right\vert\leq \frac{2}{n(n-1)},\quad \sum_{i<j} \mathbb{E}\left(\frac{2}{n(n-1)}(W_{ij}-p)\right)^2=\frac{2p(1-p)}{n(n-1)}. $$
We know from the Bernstein's inequality for bounded distributions \cite[Theorem 2.8.4]{vershynin2018high} that for each $t\ge0$ one has
	\begin{align}\label{Berns-ineq}
		P\left(|\mu - p| \geq t\right) \leq 2 \exp\left(-\frac{t^2/2}{\frac{2p(1-p)}{n(n-1)} + \frac{2t}{3n(n-1)}} \right).
	\end{align}
 Let $t = c_0 \sqrt{p}\log n/n$. By the facts that $n-1 \geq \frac{n}{2}$ and $np \geq \log n$ when $n$ is large enough, it follows that 
$$\frac{t^2/2}{\frac{2p(1-p)}{n(n-1)} + \frac{2t}{3n(n-1)}}=\frac{3c_0^2 p \log^2 n\cdot (n-1)}{12np(1-p) + 4c_0 \sqrt{p}\log n}\geq\frac{3c_0^2  \log^2 n}{24(1-p) + 8c_0 \sqrt{p}} \geq \frac{c_0^2 \log^2 n}{8+ 8c_0/3},$$
and combining this with \eqref{Berns-ineq} implies that
\begin{align*}
	P\left(|\mu - p| \geq c_0 \sqrt{p}\log n/n \right) \leq 2 \exp\left(-\frac{c_0^2 \log^2 n}{8+ 8c_0/3}\right).
\end{align*}
Since $c_0$ is set to sufficient large, we obtain $P\left(|\mu - p| \geq c_0 \sqrt{p}\log n/n \right) \leq n^{-10}$, which implies that \eqref{eq: bernmean} holds with probability at least $1-n^{-10}$.
As a result, also with probability at least $1-n^{-10}$, it follows that
\begin{align} \label{eq: mupp2}
	|\mu - p | \leq  c_0 \sqrt{p}\log n/n \leq \frac{p}{2}
\end{align}  as long as $c_0\leq\frac{np}{2\sqrt{p} \log n}$ (such $c_0$ always exists when $n$ is enough large since $\sqrt{n}p\geq \kappa_0$).
	

	Note that from \eqref{eq: agstargstarnorm} it follows that
	\begin{align}\label{key-triangle-ineq}
		  \| (A-{\mu\cdot 1_{nd}1_{nd}^\top}) \circ  G^{\star}G^{\star \top} \| 
		&=  \|W - {\mu \cdot 1_n1_n^\top}\| \notag\\
		& \leq \|W-\mu I_n -  ( {p \cdot 1_n1_n^\top}-pI_n)\| + |\mu-p|\cdot \|1_n1_n^\top - I_n \|  \notag\\
		& \leq \|W-\mu I_n - \mathbb{E}(W-\mu I_n)\| + n\cdot |\mu-p|.  
	\end{align} 	
	Also, note that from \cite[Lemma 5.1]{gao2020exact} it follows that
	\begin{equation}\label{prob-key1}
	P\left(\|W-\mu I_n - \mathbb{E}(W-\mu I_n)\|\leq c_0\sqrt{np}\right)\ge 1-n^{-10},
	\end{equation}
	and from the fact that \eqref{eq: bernmean} holds with probability at least $1-n^{-10}$,  $\log n\leq \sqrt{n}$ we know that
	\begin{equation}\label{prob-key2}
		P\left(n\cdot |\mu-p|>c_0\sqrt{np}\right)\leq P\left(|\mu-p|>c_0 \sqrt{p}\log n/n\right)\leq n^{-10}.
	\end{equation}
Thus, we obtain by the union bound together with \eqref{key-triangle-ineq}, \eqref{prob-key1} and \eqref{prob-key2}  that
\begin{align}\label{key-triangle-ineq-2}
	&P\left(\| (A-{\mu\cdot 1_{nd}1_{nd}^\top}) \circ  G^{\star}G^{\star \top} \|>2c_0\sqrt{np}\right) \notag\\
  & \leq P\left(\|W-\mu I_n - \mathbb{E}(W-\mu I_n)\|>c_0\sqrt{np}\right) + P\left(n\cdot |\mu-p|>c_0\sqrt{np}\right)  \notag\\
  & \leq  2n^{-10}.
\end{align}	
Consequently we have
\begin{align}\label{key-dis-p}
&P\left(\| (A-{\mu\cdot 1_{nd}1_{nd}^\top}) \circ  G^{\star}G^{\star \top} \|\leq 2c_0\sqrt{2n\mu}\right)\notag\\
&=  1-P\left(\| (A-{\mu\cdot 1_{nd}1_{nd}^\top}) \circ  G^{\star}G^{\star \top} \Vert> 2c_0\sqrt{2n\mu}\right)\notag\\
&\geq  1-P\left(\| (A-{\mu\cdot 1_{nd}1_{nd}^\top}) \circ  G^{\star}G^{\star \top} \|>2c_0\sqrt{np}\right)-P(p>2\mu)\notag\\
&\geq 1-3n^{-10}
\end{align}
because from 
\eqref{eq: mupp2} one has
	$2\mu\geq p$ with probability at least $1-n^{-10}$.


	Let $N$ be an $\frac{1}{2}$-net of $\{x\in\mathbb{R}^d \mid  \Vert x\Vert_2=1\}$. It is known that the net $N$ can be chosen with $|N|\leq 6^d$. By the definition of the operator norm as well as Fact \ref{fact: epsilonnet} for $i \in [n]$, it follows that
	\begin{align}\label{property3-keyineq}
		 \left\| \left((A - {\mu\cdot 1_{nd}1_{nd}^\top} ) \circ G^\star G^{\star\top} \right)_i^\top \hat{G} \right\| 
		 &= \left\| \sum_{ j \in [n] \setminus \{i\} }  (W_{ij}-\mu) G_i^\star G_j^{\star \top}\hat{G}_j \right\|= \left\| \sum_{ j \in [n] \setminus \{i\} }  (W_{ij}-\mu) G_j^{\star \top}\hat{G}_j \right\|\notag\\
		&\leq 4 \max_{u \in N, v \in N}\left| \sum_{ j \in [n] \setminus \{i\} }  (W_{ij}-p) \cdot u^\top G_j^{\star \top}\hat{G}_j v \right| + n\cdot|\mu-p|.
	\end{align}
	Next, we are going to get an upper bound of $\max_{u \in N, v \in N}\left| \sum_{ j \in [n] \setminus \{i\} }  (W_{ij}-p) \cdot u^\top G_j^{\star \top}\hat{G}_j v \right|$.
	From the definition we know that for any $u\in N$, $v \in N$ and $i\in[n]$, 
	$$\sum_{ j \in [n] \setminus \{i\} }  (W_{ij}-p) \cdot u^\top G_j^{\star \top}\hat{G}_j v$$ is a sum of independent mean-zero random variables, and it follows that $$\left\vert(W_{ij}-p) \cdot u^\top G_j^{\star \top}\hat{G}_j v\right\vert\leq 1,\quad \sum_{j \in [n] \setminus \{i\}} \mathbb{E}\left((W_{ij}-p) \cdot u^\top G_j^{\star \top}\hat{G}_j v\right)^2\leq np. $$
Then for any $u\in N$, $v \in N$ and $i\in[n]$, again from Bernstein's inequality for bounded distributions, we have for each $t\ge0$ that
	\begin{align*}
		P \left(\left| \sum_{ j \in [n] \setminus \{i\} }  (W_{ij}-p) \cdot u^\top G_j^{\star \top}\hat{G}_j v \right| > t \right) \leq 2 \exp \left(\frac{-t^2/2}{np + t/3} \right),
	\end{align*}
and consequently it follows from the union bound that 
	\begin{align*}
		& P\left( \max_{u \in N, v \in N}\left| \sum_{ j \in [n] \setminus \{i\} }  (W_{ij}-p) \cdot u^\top G_j^{\star \top}\hat{G}_j v \right| > t/4 \right) \leq  2 \times 36^d \exp \left(\frac{-t^2/32}{np + t/12} \right).
	\end{align*}
	Applying again the union bound and with the help of \eqref{property3-keyineq} for $i \in [n]$, we have 
	\begin{align}\label{key3-contr}
		&P\left(  \max_{i \in [n]} \left\| \left((A - {\mu\cdot 1_{nd}1_{nd}^\top} ) \circ G^\star G^{\star\top} \right)_i^\top \hat{G} \right\| \leq t \right)\notag\\ 
		&\ge 1-\sum_{i=1}^n P\left( \left\| \left((A - {\mu\cdot 1_{nd}1_{nd}^\top} ) \circ G^\star G^{\star\top} \right)_i^\top \hat{G} \right\| > t \right)\notag\\ 
		&\geq 1-\sum_{i=1}^n P\left( \max_{u \in N, v \in N}\left| \sum_{ j \in [n] \setminus \{i\} }  (W_{ij}-p) \cdot u^\top G_j^{\star \top}\hat{G}_j v \right| > t/8 \right)- \sum_{i=1}^n P\left( n\cdot|\mu-p|>t/2 \right) \notag\\
		& \geq  1-  2 \left(36\exp \left(\frac{-t^2/128}{ d(n p+t/24)} + \frac{\log n}{d} \right) \right)^d - 2n \exp\left(-\frac{t^2/8}{\frac{2np(1-p)}{n-1} + \frac{t}{3(n-1)}} \right).
	\end{align}
	By setting $ t = c_0 \sqrt{n d p \log n }$ for \eqref{key3-contr} and the fact that $\lim \limits_{n \rightarrow \infty} \frac{np}{d \log n} \rightarrow \infty$ (obtained from the assumption $p \geq \frac{\kappa_0d}{\sqrt{n}}$), we arrive at
		\begin{align*}
			P\left(  \max_{i \in [n]} \left\| \left((A - {\mu\cdot 1_{nd}1_{nd}^\top} ) \circ G^\star G^{\star\top} \right)_i^\top \hat{G} \right\| \leq c_0 \sqrt{n d p \log n } \right) \geq 1-n^{-10} .
		\end{align*}
		 Then again from $2\mu\geq p$ with probability at least $1-n^{-10}$, with the similar argument of \eqref{key-dis-p} we know that \eqref{Concen-key2} holds with probability at least $1-2n^{-10}$.
	
	Note that since $Z_{ii}=\bm{0}$, $i\in [n]$, we can assume $A_{ii}= \bm{0}$, $i\in [n]$ without changing the value of $\|A \circ Z\|$.
	Specifically, we use $P_A$ for conditional probability $P(\cdot|A)$. Define the event
	\begin{align}\label{eq: conditionevent}
		\mathcal{A}: = \left\{\max_{i \in [n]}  \sum_{j \in [n]} W_{ij} \leq 2np \right\}.
	\end{align}
	Under the fact that $\lim\limits_{n \rightarrow \infty}\frac{np}{d \log n}= \infty $, 
	with the similar analysis for proving \eqref{eq: bernmean} through the union bound  and Bernstein's inequality for bounded distributions, we have that
	$$P(\mathcal{A}^c)\leq \sum_{i\in[n]} P\left(\sum_{j \in [n]} W_{ij} > 2np\right)=\sum_{i\in[n]} P\left(\sum_{j \in [n]} \left(W_{ij}-p\right) > np\right)\leq n^{-10}.$$ Since $|W_{ij}|\leq 1$ for any $i,j\in[n]$ and $A=W \otimes (1_d 1_d^\top)$, we know from Lemma \ref{lemma: aop} that there exists $c_1>0$ for any $t\ge0$ it follows that
	\begin{align*}
		\sup_{A \in \mathcal{A}} P_A\left(\| A \circ Z \| \geq c_1\sqrt{ndp} + t\right) \leq \exp(-t^2/2).
	\end{align*}
   This implies that for sufficient large $c_0$ and $t=(c_0-c_1)\sqrt{ndp}$ one has that
	\begin{align*}
		\sup_{A \in \mathcal{A}} P_A\left(\| A \circ Z\| \geq c_0\sqrt{ndp} \right) \leq n^{-10}.
	\end{align*}
	Therefore, we conclude that
	\begin{align*}
		P\left(\| A \circ Z\| \geq c_0\sqrt{ndp}\right) \leq P(\mathcal{A}^c) +\sup_{A \in \mathcal{A}} P_A\left(\| A \circ Z\| \geq c_0\sqrt{ndp} \right) \leq 2n^{-10}.
	\end{align*}
Combining it with \eqref{eq: mupp2} would result to $P\left(\| A \circ Z\| \leq c_0\sqrt{2nd \mu}\right) \geq 1-3n^{-10}$ with the similar argument of \eqref{key-dis-p}. Thus, from the union bound we know that \eqref{eq: bernmean}-\eqref{Concen-key3} hold with probability at least $1-9n^{-10}$. The proof is complete.
\end{proof}

Now, we are ready for the proof of Corollary \ref{cor: Gaussianlinear}.
\begin{proof}[Proof of Corollary \ref{cor: Gaussianlinear}]
It suffices to verify the assumptions (i)-(iii) in Theorem \ref{thm:convergence-rate}. From \eqref{eq: agstargstarnorm}, \eqref{Concen-key1}, \eqref{Concen-key3}, and \eqref{eq: mupp2} we know that with high probability
	\begin{align}
	\|W - {\mu \cdot 1_n1_n^\top} \| + \|  A \circ \Delta\| 
	&=\| (A-{\mu\cdot 1_{nd}1_{nd}^\top}) \circ G^{\star}G^{\star \top}  \|+\sigma\Vert A\circ Z\Vert\notag\\
	&\leq 2c_0 \sqrt{2n \mu}+\frac{\kappa_1 n^{1/4}p^{1/2}}{d}\cdot c_0 \sqrt{2nd\mu} \leq \frac{n^{3/4}\mu}{60 d^{1/2}}.\label{cor-con1}
	\end{align}
	In order to estimate the term $\| (A \circ \Delta)G^\star  \|_\infty$, we need a bound for the term $\| \sum_{j\neq i} W_{ij}Z_{ij} \|$.
	Recall that we can have $P(\mathcal{A}^c) \leq n^{-11}$ where the event $\mathcal{A}$ is defined in \eqref{eq: conditionevent}.	
Since each element of $\sum_{j\neq i} W_{ij}Z_{ij}$ is the sum of i.i.d. Gaussian variables, we know for each $k,l\in[d]$ that $\left(\sum_{j\neq i} W_{ij}Z_{ij}\right)_{kl}$ is a mean-zero Gaussian random variable with 
variance $\sigma_{kl}^2=\sum_{j\neq i} W_{ij}^2=\sum_{j\neq i} W_{ij}\leq 2np$ under the event $\mathcal{A}$. Thus, for any $A\in\mathcal{A}$ we know that for any $t>0$, with probability at least $1-2\exp(-t^2)$,
	$$\left\| \sum_{j\neq i} W_{ij}Z_{ij} \right\|\leq c_2\cdot \max_{k,l} \left\Vert\left(\sum_{j\neq i} W_{ij}Z_{ij}\right)_{kl}\right\Vert_{\psi_2} (\sqrt{d}+t)\leq c_3 \cdot \sqrt{np} (\sqrt{d}+t) $$
	for some constant $c_2, c_3>0$ according to \cite[Theorem 4.4.5, Example 2.5.8]{vershynin2018high}, where $\Vert\cdot\Vert_{\psi_2}$ is the sub-Gaussian norm defined in \cite[Definition 2.5.6]{vershynin2018high}.
	Thus, let $t=c_0\sqrt{d\log n}$ for sufficient large $c_0$, and for some constant $c_4>0$ it follows that
	\begin{align*}
		\sup_{A \in \mathcal{A}} P_A\left( \left \| \sum_{j\neq i} W_{ij}Z_{ij} \right\| \geq c_4 \sqrt{ndp\log n}  \right) \leq n^{-11},
	\end{align*}
and consequently
	\begin{align*}
		P\left(\left\| \sum_{j\neq i} W_{ij}Z_{ij}\right\| \geq c_4\sqrt{ndp\log n}\right) \leq P(\mathcal{A}^c) +\sup_{A \in \mathcal{A}} P_\mathcal{A}\left( \left \| \sum_{j\neq i} W_{ij}Z_{ij} \right\| \geq c_4 \sqrt{ndp\log n}  \right) \leq 2 n^{-11} .
	\end{align*}
Thus, we readily obtain from  $2\mu\geq p$  with probability at least $1-n^{-10}$ (from \eqref{eq: mupp2}), $0 <\sigma \leq \frac{\kappa_1 n^{1/4}p^{1/2}}{d} $ and the union bound that, 
	\begin{align*}
	    \| (A \circ \Delta)G^\star  \|_\infty  = \sigma \cdot\max\limits_{i\in [n]} \left\| (A \circ Z)^\top_i G^\star\right\| 
		= \sigma \cdot \max_{i\in [n]} \left\| \sum_{j\neq i} W_{ij}Z_{ij}G_j^\star \right\| 
		&= \sigma \cdot \max_{i\in [n]} \left\| \sum_{j\neq i} W_{ij}Z_{ij} \right\| \notag\\
		&\leq \frac{\kappa_1 n^{1/4}p^{1/2}}{d} \cdot c_4\sqrt{ndp\log n}  \leq \frac{n\mu}{10},
	\end{align*}
holds with high probability, where the third equality is because we can set the ground truth $G_j^\star$, $j\in[n]$ to be $I_d$ without loss of generality for Gaussian noise setting.
The assumption (ii) is directly from \eqref{Concen-key2}, and the assumption (iii) is from  \eqref{cor-con1} with $\frac{\kappa_0 n^{3/4}\mu}{d^{1/2}} \leq \alpha \leq \frac{\kappa_1 np}{d^{1/2}}$ and $2\mu\ge p$  with probability at least $1-n^{-10}$. The proof is complete.
\end{proof}

\section{Conclusions}
In this work, we investigate the characterizations of orthogonal group synchronization problem which is nonconvex but with nice properties, and derive the local error bound property with general additive noise models under incomplete measurements. As an algorithmic consequence,  the linear convergence result of the GPM (with spectral initialization) to the global maximizers
is proved. For future directions, it would be interesting to study the convergence properties of different algorithms through the derived error bound result since it is an algorithm-independent property only related to the synchronization problem itself. Also, though the study for phase synchronization (i.e., orthogonal synchronization on $\mathcal{SO}(2)$) \cite{boumal2016nonconvex,liu2017estimation,zhong2018near} is already quite comprehensive, it is still challenging when extended to $\sod$ ($d\ge3$) without vital properties of $\mathcal{SO}(2)$ (e.g. commutativity). Due to the limited tolerance on the noise, current results for $\sod$ still need to be improved, or we may propose other effective algorithms with better and more robust convergence guarantees. 

\bibliographystyle{abbrv}
\bibliography{journalref}

\end{document}